\documentclass[11pt]{article}
\usepackage[margin=1in]{geometry}
\usepackage{amsmath,amssymb,amsthm}
\usepackage{tikz}
\usepackage{hyperref}

\usepackage[inline]{enumitem}
\usepackage{float}
\usepackage{xspace,nicefrac}
\usepackage[labelsep=none,format=plain,justification=centering]{caption}

\allowdisplaybreaks

\theoremstyle{plain}
\newtheorem{theorem}{Theorem}
\newtheorem*{lemma*}{Lemma}
\newtheorem{lemma}{Lemma}
\newtheorem{remark}{Remark}

\def\dfrac#1#2{\lower0.15ex\hbox{\large$\frac{#1}{#2}$}}

\newcommand{\hide}[1]{}
\newcommand{\ignore}[1]{}

\newcommand{\es}{\emptyset}

\newcommand{\cG}{\mathcal{G}}
\newcommand{\cM}{\mathcal{M}}
\newcommand{\Gnd}[1][n,d]{\cG_{#1}}
\newcommand{\Mnd}[1][n,d]{\cM_{#1}}
\newcommand{\Kd}[1][d+1]{\ensuremath{K_{#1}}\xspace}

\newcommand{\rdown}[1]{\ensuremath{\lfloor #1 \rfloor}}

\definecolor{brown}{cmyk}{0, 0.72, 1, 0.45}
\definecolor{grey}{gray}{0.5}

\newcommand{\ds}{$\Delta$-switch}
\newcommand{\dsp}{$\Delta^+$-switch}
\newcommand{\dsm}{$\Delta^-$-switch}
\newcommand{\Dp}{$\Delta^+$}
\newcommand{\Dm}{$\Delta^-$}
\newcommand{\Nb}{{\rm N}}

\newcommand\Out{{\rm Out}}
\newcommand\In{{\rm In}}
\newcommand\id{{\rm id}}
\newcommand\od{{\rm od}}
\newcommand\ol[1]{\overline{ #1}}
\newcommand\sm{\setminus}
\newcommand\seq{\subseteq}
\newcommand{\dg}{\textrm{d}'}
\newcommand{\bN}{\ol{\Nb}}
\newcommand{\dist}{\textrm{dist}}

\tikzset{ b/.style = { circle 
                     , draw
                     , thick
                     , inner sep = 0pt
                     , fill = black
                     , minimum size = 3.5pt
                     }}
\tikzset{empty/.style={rectangle,draw=none,fill=none}}

\date{27 July 2021}

\title{A triangle process on regular graphs}
\author{%
Colin Cooper\\
{ Department of Informatics}\\
{ Kings College}\\
{ London WC2R 2LS, U.K.}\\
{ \texttt{colin.cooper@kcl.ac.uk}}\\
\and
Martin Dyer\thanks{Work supported by
    EPSRC grant EP/S016562/1,``Sampling in hereditary classes''.}\\
{ School of Computing}\\
{ University of Leeds}\\
{ Leeds LS2 9JT, U.K.}\\
{ \texttt{m.e.dyer@leeds.ac.uk}}\\
\and
Catherine Greenhill\thanks{Work supported by Australian Research Council grant DP190100977.}\\
{ School of Mathematics and Statistics}\\
{ UNSW Sydney}\\
{ NSW 2052, Australia}\\
{ \texttt{c.greenhill@unsw.edu.au}}
}

\begin{document}

\maketitle

\begin{abstract}
Switches are operations which make local changes to the edges of a graph, usually with the aim of preserving the vertex degrees. We study a restricted set of switches, called triangle switches. Each triangle
switch creates or deletes at least one triangle.  Triangle switches can be
used to define Markov chains which generate graphs with a given degree sequence
and with many more triangles (3-cycles) than is typical in a  uniformly random graph
with the same degrees. We show that the set of triangle switches connects
the set of all $d$-regular graphs on $n$ vertices, for any $d\geq 3$.
Hence, any Markov chain which assigns positive probability to all triangle
	switches is irreducible on $\mathcal{G}_{n,d}$ for all $d\geq 3$.
We also investigate this question for 2-regular graphs.
\end{abstract}

\section{Introduction}\label{sec:intro}

Generating graphs at random from given classes and distributions has been the subject of considerable research.
See, for example,~\cite{ABLMO,AK,BKS,CoDyGr07,CDGH,FGMS,GW,KTV,LMM,MES,MS,TikYou}. Generation using Markov chains has been a topic of specific interest in this context, in particular Markov chains based on switches of various types, for example~\cite{AK,CoDyGr07,CDGH,FGMS,KTV,LMM,MES,MS,TikYou}. Switches delete a pair of edges from the graph and insert a different pair on the same four vertices. They have the important property that they preserve the degree sequence of the graph. Thus they are useful for generating regular graphs, or other graphs with a given degree sequence. Markov chains also give a dynamic reconfigurability property, which is useful in applications, for example~\cite{CoDyGr07,FGMS,MS}.
For any such Markov chain, two questions arise. First, can it generate any graph in the chosen class?  (Formally: is the Markov chain irreducible?) Secondly, we might wish to estimate its rate of convergence to the chosen distribution. Formally: what is the mixing time of the chain?

In the applied field of social networks, the existence of triangles (3-cycles) is
seen as an indicator of mutual friendships~\cite{GKM,JGN}.  However, many random graph
models, or processes for producing random graphs, will tend to produce graphs with few
triangles. This is true for any process which generates graphs with a given degree sequence (approximately) uniformly at random, if the degree sequence gives sparse graphs. For example, the expected number of triangles is constant for $d$-regular graphs, when $d$ is constant~\cite{bollobas}.  In this paper, we study a restricted set of switches, called
\emph{triangle switches}, and consider any reversible Markov chain whose transitions
are exactly the triangle switches. We
answer the first question (``is the Markov chain irreducible?'') for such a chain
on the state space of $d$-regular graphs, for any $d$.  Note that the answer to this
question is independent of the probabilities assigned to each triangle switch by the
Markov chain, as it is really a property of the undirected graph underlying the
Markov chain.  Hence we do not specify precise transition probabilities in this paper.

We leave the mixing question for future research, noting only that tight bounds on mixing time seem hard to come by in this setting. The recent paper~\cite{TikYou} is a notable exception.

Triangle switches were introduced in~\cite{CDG19} in the context of cubic graphs.
In~\cite{CDG19}, some Markov chains using triangle switches were defined, with transition probabilities assigned to encourage the formation of triangles. After proving that triangle switch chains are irreducible for 3-regular graphs, an analysis was
given~\cite[Section~4]{CDG19} showing that it is possible to generate (cubic) graphs using this approach which have $\Omega(n)$ triangles in $O(n)$ steps of the Markov chain.

The proofs in~\cite{CDG19} do not easily generalise to regular graphs of arbitrary $d$,
though the main approach in our proof of irreducibility comes from~\cite{CDG19}.
If a component of a $d$-regular graph is a clique (that is, a complete subgraph) then it
must be isomorphic to $K_{d+1}$.  We call such a component a \emph{clique component}.
Our approach is to show that starting from an arbitrary $d$-regular graph,
triangle switches can be used to increase the number of clique components.
Furthermore, we show how to alter the set of vertices in a given clique component
using triangle switches.
After creating as many clique components as possible, there is at most one
additional component $C$, which must satisfy
$d+1 < |C| < 2(d+1)$.  We call such a component a \emph{fragment}.
We prove that triangle switches connect the set
of all fragments on a given vertex set.
In the cubic case, this last step is
simpler as the only possible fragments are $K_{3,3}$ and $\bar{C}_6$.

Our result can be viewed as solving a particular \emph{reconfiguration} problem for regular graphs.  Reconfiguration is a topic of growing interest in discrete mathematics. For an introduction to the topic, and a survey of results, see~\cite{Nim}. We note that reconfiguration problems can be as hard as \textsf{PSPACE}-complete, in general. Our results show that there is a polynomial time algorithm to construct a path of triangle swtiches between any two $d$-regular graphs on $n$ vertices.

The plan of the paper is as follows. In Section~\ref{sec:defs} we collect relevant notation and definitions.
In~Section~\ref{subsec:switches}, we define and review switches and restricted switches, in particular triangle switches, and state our main result,
Theorem~\ref{thm:irreducible}.  For most of the paper we assume that $d\geq 3$.
In Section~\ref{sec:small} we show that the set of all fragments with a given vertex set is connected under triangle switches. In Section~\ref{sec:split}, we show that triangle switches can be used to
create a clique component, starting from any $d$-regular graph with at least $2(d+1)$ vertices.
In Section~\ref{sec:sorting} we show how to relabel the vertices in clique components
using triangle switches, and hencecomplete our proof of irreducibility. Finally, in Section~\ref{sec:d=2}, we consider the irreducibility question for $d$-regular graphs with $d\leq 2$.

\subsection{Definitions, notation and terminology} \label{sec:defs}

The notation $[k]$ will denote the set $\{1,2,\ldots,k\}$, for any integer $k$.
For other graph-theoretic definitions and concepts not given here, see~\cite{West}, for example.

Given a set $V$ of vertices, let $V^{\{2\}}$ be the set of unordered pairs of distinct
elements from $V$.  A graph $G=(V,E)$ on vertex set $V$ has edge set $E\seq V^{\{2\}}$.
We usually denote $|V|$ by $n$. We use the notation $xy$ as a shorthand for the unordered pair $\{x,y\}$, whether or not this pair is an edge.
If $E'\seq E$ and $V'=\{v\in V : v\in e\in E'\}$, then $G'=(V',E')$ is a \emph{subgraph} of $G$. Given any vertex subset $U\subseteq V$, the subgraph $G[U]$ \emph{induced} by~$U$ has vertex set $U$ and edge set $E'=U^{(2)}\cap E$. If $|U|=k$ and $G[U]$ is a $k$-cycle, then we say that $G[U]$ is an \emph{induced $C_k$}.

We will write $G\cong H$ to indicate that graphs $G$ and $H$ are isomorphic.
Given a graph $G=(V,E)$, the \emph{complement} of $G$ is the graph $\ol{G}=(V,\ol{E})$
with $\ol{E}=V^{\{2\}}\sm E$. An edge of $\ol{G}$ will be called a \emph{non-edge} of $G$.

The \emph{distance} $\dist(u,v)$ between two vertices $u$ and $v$ is the number of edges
in a shortest path from $u$ to $v$ in $G$, with $\dist(u,v):=\infty$ if no such
path exists. The maximum distance between two vertices in $G$ is the \emph{diameter} of $G$
and $G$ is connected if it has finite diameter.
The \emph{component} $C$ of $G$ containing $v$ is the
largest connected induced subgraph of $G$ which contains $v$.

Given a graph $G=(V,E)$ and vertex $v\in V$,
let $\Nb_G(v)=\{u:uv\in E\}$ denote the neighbourhood of $v$,
and let $\deg_G(v)=|\Nb_G(v)|$ denote the degree of $v$ in $G$.
The \emph{closed neigbourhood} of $v$ is $\Nb_G[v]:=\Nb_G(v)\cup\{v\}$.
We sometimes drop the subscript and write $\Nb(v)$ or $\Nb[v]$.

Say that $G$ is \emph{regular} if every vertex has the same degree, and
if $\deg_G(v)=d$ for all $v\in V$ then we say that $G$ is $d$-regular.
Let $\Gnd$ be the set of all $d$-regular graphs with vertex set $V=[n]$.
Note that $\Gnd$ is non-empty if and only if either $d$ or $n$ is even.
This result seems to be folklore, but is easy to prove. Necessity is implied by edge counting, and suffiency by a direct construction. An indirect proof can be found in~\cite[Prop.\,1]{TriTy}. As usual, $\Kd\in\Gnd[d+1,d]$ denotes the complete graph on $d+1$ vertices, and $K_{d,d}\in\Gnd[2d,d]$ denotes the complete bipartite graph on $d+d$ vertices. A graph in $\Gnd$ with $d+1<n<2(d+1)$ will be called a \emph{fragment}. Note that $\Kd$ is not a fragment.

We will often regard a graph $G$ as \emph{layered}, in the following way. Let $v$ be a given (fixed) vertex of a $d$-regular $G=(V,E)$, where $n=|V|\geq 2(d+1)$, and let $C\seq V$ determine the component $G[C]$ of $G$ such that $v\in C$. We regard $C$ as partitioned by edge-distance from $v$, with $V_i=\{u\in C:\dist(v,u)=i\}$.
Thus $V_0=\{v\}$ and $V_1=\Nb(v)$, so $|V_0|=1$ and $|V_1|=d$. Since $V_2$ appears frequently in the proof, we will denote $|V_2|$ by $\ell$.
By definition, $C$ is a disjoint union $C=\bigcup_{i\geq0}V_i$, and $|C|=\sum_{i\geq0}|V_i|$. Let $G_i=G[V_i]=(V_i,E_i)$, and note that $G_0=(\{v\},\es)$. Let $\Nb'(u)$ be the neighbourhood of $u\in V_i$ in $G_i$, i.e. $\Nb'(u)=\Nb(u)\cap V_i$, and let $\dg(u)=|\Nb'(u)|$ be the degree of $u$ in $G_i$. We omit explicit reference to $i$ in this notation, since it is implicit from $u\in V_i$.
Given $u\in V_i$, we denote the set of non-neighbours of $u$ in $G_i$ by $\bN'_i(u)$.

We will regard the edges of $G[C]$ from $V_i$ to $V_{i+1}$ as being directed. Then, for $u \in V_i$, $\In(w)$ is the neighbour set of $u$ in $V_{i-1}$ and $\Out(u)$ is the neighbour set of $u$ in $V_{i+1}$. Thus, if $u\in V_i$, $\In(u)=\Nb(u)\cap V_{i-1}$ and $\Out(u)=\Nb(u)\cap V_{i+1}$. Then let $\id(u)=|\In(u)|$ be the in-degree of $u\in V_i$ and  $\od(u)=|\Out(u)|$ the out-degree of $u$. Thus $\Nb(u)=\Nb'(u)+\In(u)+\Out(u)$, and $\dg(u)+\id(u)+\od(u)=d$. In particular, $\dg(v)=\id(v)=0$, and $\od(v)=d$. If $u\in V_1=\Nb(v)$, $\id(u)=1$ and so $\dg(u)+1+\od(u)=d$, and thus $\od(u)=d-1-\dg(u)=|\bN'_1(u)|$.

A pair of vertices $x,y \in V_i$ will be said to be \emph{below} a pair $a,b\in V_{i+1}$ if  $a \in \Out(x)$, $b \in \Out(y)$, and $a,b$ will be said to be \emph{above} $x,y$. Note that, if $a,b\in V_{i+1}$ is not above some pair $x,y\in V_i$, there must be a unique $x \in V_i$ with $a,b\in \Out(x)$. We will be most interested in the case where $i=1$ and $ab\notin E_2$.

\subsection{Switches}\label{subsec:switches}
As described above, an established approach to the generation of graphs with given degrees is to use local edge transformations known as switches. The process is \emph{irreducible} if any graph in the class can be obtained from any other by a sequence of these local transformations. Here we will consider three possibilities for this local transformation.

In a switch, a  pair of edges $xy$, $wz$ of graph $G=(V,E)$ are chosen at random in some fashion, and replaced with the
pair $xw$ $yz$, provided these are currently non-edges. We make no other assumptions about the subgraph $G[\{w,x,y,z\}]$.  See Figure~\ref{fig:switch}. Taylor~\cite{taylor} showed irreducibility of this process. (See also ~\cite[Thm.\,1.3.33]{West}, where switches are called ``2-switches'', and~\cite{LMM} for a more constructive proof.) Cooper, Dyer and Greenhill~\cite{CoDyGr07} showed rapid mixing for regular graphs, and a generalisation to some (relatively sparse) irregular degree sequences was given in~\cite{GS}. Clearly switches preserve vertex degrees, since each vertex in the switch has one edge deleted and one added, and all other vertices are unaffected. Switches can easily be restricted to preserve bipartiteness, by requiring that $\{w,y\}$ (or equivalently $\{x,z\}$) belong to the same side of the bipartition. In fact, switches were first used as the transitions of a Markov chain for bipartite graphs~\cite{KTV}.
\begin{figure}[H]
\begin{center}
\begin{tikzpicture}[scale=0.7]
\draw [thick,-] (4,0) -- (2,0) (2,2) -- (4,2);
\draw [dashed,thick] (4,0)--(4,2) (2,0)--(2,2) ;
\draw [fill] (2,0) circle (0.1);
\draw [fill] (2,2) circle (0.1);
\draw [fill] (4,2) circle (0.1); \draw [fill] (4,0) circle (0.1);
\node [above] at (2,2.1)  {$x$}; \node [below] at (2,-0.1)  {$w$};
\node [above] at (4,2.1)  {$y$}; \node [below] at (4,-0.1)  {$z$};
\draw  (6,1.1)edge[line width=1.5pt,<->](8,1.1);
\draw (7,1.1)node[empty,label=above:switch] {} ;
\begin{scope}[shift={(8.3,0)}]
\draw [thick,-] (2,2) -- (2,0) (4,0) -- (4,2);
\draw [thick,dashed,-] (4,0) -- (2,0) (2,2) -- (4,2);
\draw [fill] (2,0) circle (0.1);
\draw [fill] (2,2) circle (0.1);
\draw [fill] (4,2) circle (0.1); \draw [fill] (4,0) circle (0.1);
\node [above] at (2,2.1)  {$x$}; \node [below] at (2,-0.1)  {$w$};
\node [above] at (4,2.1)  {$y$}; \node [below] at (4,-0.1)  {$z$};
\end{scope}
\end{tikzpicture}
\caption{\;\; A switch}\label{fig:switch}
\end{center}
\end{figure}
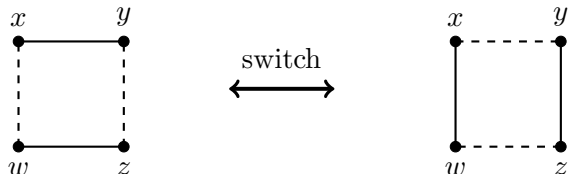
If we wish to generate only \emph{connected} graphs, we may use the \emph{flip}. This is defined in the same way as the switch, except that we specify that $wy$ must also be an edge. See Figure~\ref{fig:flip}. Note that a flip is a restricted form of switch which cannot disconnect the graph. Mahlmann and Schindelhauer~\cite{MS} showed irreducibility of flips, and Cooper, Dyer, Greenhill and Handley~\cite{CDGH} showed rapid mixing for regular graphs of even degree. Note that flips are not well-defined on bipartite graphs, since $\{w,y\}$ clearly cannot be on the same side of a bipartition.
Mahlmann and Schindelhauer also considered other restricted forms of switch,
where there must be a $k$-edge path between $w$ and $y$.  The flip chain
corresponds to $k=1$, while the ``2-flipper'' with $k=2$ preserves
connected bipartite graphs.

Irreducibility of the 2-flipper was proved in ~\cite{MS}, but the idea does not seem to have been considered subsequently.

\begin{figure}[H]
\begin{center}
\begin{tikzpicture}[scale=0.7]
\draw [thick,-] (4,0) -- (2,0) (2,2) -- (4,2) (2,0)--(4,2);
\draw [dashed,thick] (4,0)--(4,2) (2,0)--(2,2) ;
\draw [fill] (2,0) circle (0.1);
\draw [fill] (2,2) circle (0.1);
\draw [fill] (4,2) circle (0.1); \draw [fill] (4,0) circle (0.1);
\node [above] at (2,2.1)  {$x$}; \node [below] at (2,-0.1)  {$w$};
\node [above] at (4,2.1)  {$y$}; \node [below] at (4,-0.1)  {$z$};
\draw  (6,1.1)edge[line width=1.5pt,<->](8,1.1);
\draw (7,1.1)node[empty,label=above:flip] {} ;
\begin{scope}[shift={(8.3,0)}]
\draw [thick,-] (2,2) -- (2,0) (4,0) -- (4,2)  (2,0)--(4,2);
\draw [thick,dashed,-] (4,0) -- (2,0) (2,2) -- (4,2);
\draw [fill] (2,0) circle (0.1);
\draw [fill] (2,2) circle (0.1);
\draw [fill] (4,2) circle (0.1); \draw [fill] (4,0) circle (0.1);
\node [above] at (2,2.1)  {$x$}; \node [below] at (2,-0.1)  {$w$};
\node [above] at (4,2.1)  {$y$}; \node [below] at (4,-0.1)  {$z$};
\end{scope}
\end{tikzpicture}
\caption{: A flip}\label{fig:flip}
\end{center}
\end{figure}

In~\cite{CDG19}, a different restriction of switches was introduced, designed to ensure
that every switch changes the set of triangles in the graph.
The definition is as for switches, except that $x$ and $w$ must have a common
neighbour, which we denote by $v$.
This is a \emph{triangle switch}, which we abbreviate as \ds.
Every \ds\ makes (creates) or breaks (destroys) at least one triangle.
Again, we make no further assumption about $G[\{v,w,x,y,z\}]$.
Clearly, \ds es do not preserve bipartiteness, since bipartite graphs have no triangles.

Specifically, if the 4-edge path $yxvwz$ is present in the graph and the edges
$xw$, $yz$ are absent, the {\em make} triangle switch, denoted \dsp,
deletes the edges $xy$, $wz$ and replaces them with
edges $xw$, $yz$, forming a triangle on $v,x,w$.
The \dsp\ is illustrated in Figure~\ref{fig:fig1}, reading from left to right.
Conversely, if the edge $yz$ and the triangle on $v,x,w$ are present
in the graph, such that the edges $xy$, $wz$ are both absent, then the {\em break} triangle switch, denoted \dsm,
deletes the edges $xw$, $yz$ and replaces them with the edges $xy$, $wz$.
This destroys the triangle on $v,x,w$.
The \dsm\ is illustrated in Figure~\ref{fig:fig1}, reading from right to left. Note that a \dsm\ reverses a \dsp\ and vice versa.
\begin{figure}[H]
\begin{center}
\begin{tikzpicture}[scale=0.75]
\draw [thick,-] (4,0) -- (2,0) -- (0,1) -- (2,2) -- (4,2);
\draw [dashed,thick] (4,0)--(4,2) (2,0)--(2,2) ;
\draw [fill] (0,1) circle (0.1); \draw [fill] (2,0) circle (0.1);
\draw [fill] (2,2) circle (0.1);
\draw [fill] (4,2) circle (0.1); \draw [fill] (4,0) circle (0.1);
\node [left] at (-0.1,1)  {$v$};
\node [above] at (2,2.1)  {$x$}; \node [below] at (2,-0.1)  {$w$};
\node [above] at (4,2.1)  {$y$}; \node [below] at (4,-0.1)  {$z$};
\draw  (6,1.3)edge[line width=1.5pt,->](8,1.3);
\draw (7,1.3)node[empty,label=above:\dsp] {} ;
\draw  (6,0.8)edge[line width=1.5pt,<-](8,0.8);
\draw (7,0.9)node[empty,label=below:\dsm] {} ;
\begin{scope}[shift={(10,0)}]
\draw [thick,-] (0,1) -- (2,2) -- (2,0) -- (0,1) (4,0) -- (4,2);
\draw [thick,dashed,-] (4,0) -- (2,0) (2,2) -- (4,2);
\draw [fill] (0,1) circle (0.1); \draw [fill] (2,0) circle (0.1);
\draw [fill] (2,2) circle (0.1);
\draw [fill] (4,2) circle (0.1); \draw [fill] (4,0) circle (0.1);
\node [left] at (-0.1,1)  {$v$};
\node [above] at (2,2.1)  {$x$}; \node [below] at (2,-0.1)  {$w$};
\node [above] at (4,2.1)  {$y$}; \node [below] at (4,-0.1)  {$z$};
\end{scope}
\end{tikzpicture}
\end{center}
\caption{: The triangle switches}\label{fig:fig1}
\end{figure}
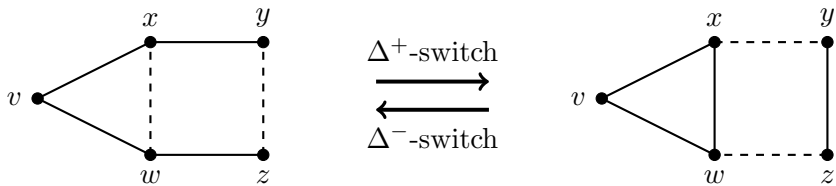
A \ds\ which involves $v$ and two incident edges, as in Figure~\ref{fig:fig1},
will be called a \ds\ at $v$.
Note that this is equivalent to a switch in the graph $H=G[V_1\cup V_2]$, if the graph is layered from $v$, and we will use this equivalence in our arguments.

Let $\Mnd$ be the graph with vertex set $\Gnd$, such that $\{G,G'\}$ an edge if and only if $G'$ can be obtained from $G$ by a single \ds. A time-homogeneous Markov chain with state space $\Mnd$ will be called a
\ds\ chain if its transition matrix $P$ satisfies $P(G,G')>0$ if and only if $\{G,G'\}\in E(\Mnd)$.  That is, $\Mnd$ is the graph underlying any \ds\ chain. Our main result is the following.

\begin{theorem}\label{thm:irreducible}
	Suppose that $d\geq3$.
	Then the graph $\Mnd$ is connected. Equivalently,
	any \ds\ chain is irreducible on $\Gnd$.
\end{theorem}

Next we show that \ds es connect the set of all fragments on a given vertex set.

\section{Small regular graphs}\label{sec:small}

We first prove some properties of fragments which will be required later.

\begin{lemma} \label{small:diameter}
Let $G$ be a $d$-regular fragment, with $d\geq 3$.
Then $G$ is connected with diameter~$2$.
\end{lemma}
\begin{proof}
	By definition of fragment, we know that
	$G\in\Gnd$ for some $n$ with $d+1 < n < 2(d+1)$.
	If $G$ is not
	connected then it must have a component with at most $d$ vertices.
	This is a contradiction, as $G$ is $d$-regular.
Choose any vertex $v\in V$, and partition $V$ by edge-distance from $v$. Thus $V_0=\{v\}$ and $V_1=\Nb(v)$. Suppose the diameter is $r\geq 3$, and let $u\in V_r$. Then $u$ has $d$ neighbours in $V_{r-1}\cup V_r$ and hence $|V_{r-1}|+|V_r|\geq d+1$. Thus $n\geq |V_0|+|V_1|+|V_{r-1}|+|V_r|\geq 1+d + d +1=2(d+1)$, a contradiction. Thus $V_3=\es$, so $G$ has diameter at most 2.  On the other hand, if $V_2=\es$ then
$G\cong K_d$, and $n=d+1$, again a contradiction. Hence $G$ has diameter exactly 2.
\end{proof}

\begin{remark}\label{rem:small}
We prove only connectedness, but fragments have higher connectivity. It is not difficult to prove 2-connectedness. For each $d$, we have examples with connectivity only $\rdown{d/2}+1$, and we believe this represents the lowest connectivity. However, since we make no use of this, we do not pursue it further here.
\end{remark}

For an even integer $d\geq 2$, we construct the graph $T_{d,d,1}$ as follows. Take a copy of $K_{d,d}$ with vertex bipartition $(A_d,B_d)$, where $A_d=\{a_i:i\in [d]\}$ and $B_d=\{b_i:i\in [d]\}$. Let $M$ be the matching $\{a_ib_i\in d/2\}$ of size $d/2$ between $A_{d/2}$ and $B_{d/2}$.
Form $T_{d,d,1}$ from the copy of $K_{d,d}$ by deleting the edges of $M$,
adding a new vertex $v$ and an edge from $v$ to each $a_i$ and $b_i$ with $i\in [d/2]$.
Then $T_{d,d,1}$ is a $d$-regular graph tripartite graph with $2d+1$ vertices and
vertex tripartition $\{v\}\cup A_d \cup B_d$.

For example, $T_{2,2,1}$ is a 5-cycle and $T_{4,4,1}$ is shown in Figure~\ref{fig:T441}.
\begin{figure}[H]
\begin{center}
\begin{tikzpicture}[scale=1.5,inner sep=0pt]
\foreach \n in {1,2,3,4} {\node (a\n) [b] at (\n,1) {};}
\foreach \n in {1,2,3,4} {\node (b\n) [b] at (\n,0) {};}
\foreach \n in {1,2,3,4} \foreach \x in {3,4}
{\draw (a\n)--(b\x);} \foreach \x in {1,2,3,4} \foreach \n in {3,4}
{\draw (a\n)--(b\x);} \draw (a2)--(b1) (a1)--(b2);
\node (v) [b] at (0,0.5) {}; \draw (a1)--(v)--(a2) (b1)--(v)--(b2);
	\node [above] at (1,1.1) {$b_1$}; \node [above] at (2,1.1) {$b_2$};
	\node [above] at (3,1.1) {$b_3$}; \node [above] at (4,1.1) {$b_4$};
	\node [below] at (1,-0.1) {$a_1$}; \node [below] at (2,-0.1) {$a_2$};
	\node [below] at (3,-0.1) {$a_3$}; \node [below] at (4,-0.1) {$a_4$};
	\node [left] at (-0.1,0.5) {$v$};
\end{tikzpicture}
\end{center}\vspace{1ex}
  \caption{: The graph $T_{4,4,1}$}\label{fig:T441}
\end{figure}
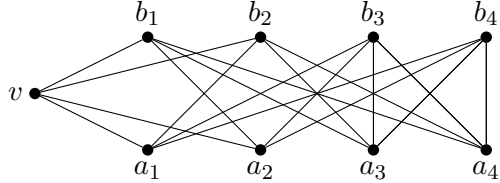
\begin{lemma}\label{small:structure}
	Suppose that $G\in\Gnd$ where $d\geq 3$ and $d+1 < n < 2(d+1)$.
	Let $ab$ be an edge of $G$.
\begin{enumerate}[topsep=0pt,itemsep=0pt]
 \item[(i)] If $n<2d$ then $G$ has a triangle which contains the edge $ab$.
 \item[(ii)] If $n=2d$ then  $ab$ is contained in a triangle or an induced $C_4$
	 in $G$.
	 Furthermore, if $G$ is triangle-free then $G\cong K_{d,d}$.
 \item[(iii)] If $n=2d+1$ then $ab$ is contained in a triangle or an induced $C_4$ in
	 $G$.
	 Furthermore, if $G$ is triangle-free then  $G\cong T_{d,d,1}$.
\end{enumerate}
\end{lemma}
\begin{proof}
 Let $A_{d-1}=\{a_i:i\in [d-1]\}=\Nb(a)\sm \{b\}$ and $B_{d-1}=\{b_i:i\in [d-1]\}=\Nb(b)\sm a$.
\begin{enumerate}[topsep=0pt,itemsep=0pt]
 \item[(i)]
 If $n<2(d-1)+2=2d$ then there exists $c\in A_{d-1}\cap B_{d-1}$.
 Hence $G$ has a triangle on the vertices $a,b,c$, which includes the edge $ab$.
 \item[(ii)] If $n= 2d$ and $a,b$ have a common neighbour, then the situation is as in (i). So suppose $A_{d-1}\cap B_{d-1}=\es$.
 Since $|V|=2d$, it follows that $V=A_{d-1}\cup B_{d-1}\cup\{a,b\}$. Now, since all vertices have degree $d$, each $a_i$ in $A_{d-1}$ must have at least one edge to a vertex in $B_{d-1}\cup\{b\}$. If some $a_i$ is adjacent to $b$ then $G$ has a triangle on the vertices $a,b,a_i$. If $a_i$ is not adjacent to $b$, but $a_ib_j$ is an edge of $G$ for some $j$,  then $abb_ja_ia$ is an $C_4$ unless $b_ja$ is an edge.  But if $ab_j$ is an edge then $G$ has a triangle on $a,b,b_j$.

 Now suppose that $G$ is triangle-free. Then $G[A_{d-1}]$ and $G[B_{d-1}]$ are independent sets.
Each $a_i\in A_{d-1}$ has $d-1$ neighbours other than $a$, and these neighbours
		must be the $d-1$ elements of $B_{d-1}$, as $G$ is triangle-free.
		Hence $G[A_{d-1}\cup B_{d-1}]\cong K_{d-1,d-1}$. Now, relabelling $a$ as $b_d$ and $b$ as $a_d$, we see that $G[A_d\cup B_d]\cong K_{d,d}$.
\item[(iii)] If $n=2d+1$ then $d$ must be even (or else $\cG_{2d+1,d}=\es$).
	The situation is similar to that of (ii), except now
	$V = A_{d-1}\cup B_{d-1} \cup \{a,b,c\}$ for some additional vertex $c$.
Again, each vertex in $A_{d-1}$ has $d-1$ neighbours other than $a$, and these
neighbours all lie in $B_{d-1}\cup \{c\}$.  The same is true with the roles of
$A_{d-1}$ and $B_{d-1}$ reversed.  But vertex $c$ has degree $d$, so there must be
at least $d-2 > 0$ edges from vertices in $A_{d-1}$ to $B_{d-1}\cup\{b\}$.
	The argument is then as in (ii).

    Now suppose that $G$ is triangle-free.
    Then $G[A_{d-1}]$, $G[B_{d-1}]$ are independent sets.
    Suppose that $c$ has $d_A$ neighbours in $A_{d-1}$ and $d_B = d-d_A$ neighbours
    in $B_{d-1}$.
    Again, let $A_d = A_{d-1}\cup \{b\}$ and $B_d = B_{d-1}\cup\{a\}$, and consider the bipartite graph $G[A_d\cup B_d]$. This graph has $d^2-d_A$ edges incident with $A_d$, and $d^2-d_B$ edges incident with $B_d$. But these must be equal, so $d_A=d_B=d/2$.
   Every vertex in $A_d$ which is not adjacent to $c$ must be adjacent to every
    vertex in $B_d$, while vertices in $A_d$ which are adjacent to $c$ must be
    adjacent to every vertex in $B_d$ except one. The same statement holds
    with the roles of $A_d$ and $B_d$ reversed.
    Hence $G\cong T_{d,d,1}$, completing the proof.\qedhere
\end{enumerate}
\end{proof}
We can strengthen this as follows.
\begin{lemma}\label{small:triangles}
Suppose $d\geq3$ and $G\in\Gnd$ is a fragment. If $ab$ is any edge in $G$ then, after at most one \ds, there is a triangle in $G'$ with $a$ as one of its vertices. Moreover, this \ds\ does not change edges incident with $b$.
\end{lemma}
\begin{proof}
	We use the notation and proof of Lemma~\ref{small:structure}. There is nothing to prove in case (i). In cases (ii) and (iii), there is nothing to prove unless
$abb_ja_ia$ is an induced $C_4$. In this situation, choose
$k\in [d-1]\setminus \{i\}$,
which is possible since $d\geq 3$. We know that $ab_j$ is a non-edge, 
as the 4-cycle is induced.
If $a_ia_k\in E$ then there is a triangle on the vertices $a,a_i,a_k$.
Otherwise, there is a \dsp\ on the 4-edge path $a_kabb_ja_i$ which creates a triangle on the vertices $a,b,b_j$, as shown in Figure~\ref{fig:smalltriangles}.
\begin{figure}[H]
\begin{center}
\begin{tikzpicture}[scale=1.5,inner sep=0pt]
\node (ai) [b,label=left:$a_i\,\,$] at  (1.5,1) {};
\node (ak) [b,label=left:$a_k\,$] at  (0,0) {};
\node (a) [b,label=below:$\strut a$] at  (1.5,0) {};
\node (bj) [b,label=right:$\,b_j$] at  (3,1) {};
\node (b) [b,label=below:$\strut b$] at  (3,0) {};
\draw (ak)--(a)--(ai) (b)--(bj) (a)--(b) (ai)--(bj);
\draw[dashed] (a)--(bj) (ai)--(ak) ;
\node at (4.5,0.75) {\Dp} ; 
\draw  (4.0,0.5)edge[line width=1.2pt, ->](5.0,0.5);
\begin{scope}[xshift=6cm]
\node (ai) [b,label=left:$a_i\,\,$] at  (1.5,1) {};
\node (ak) [b,label=left:$a_k\,$] at  (0,0) {};
\node (a) [b,label=below:$\strut a$] at  (1.5,0) {};
\node (bj) [b,label=right:$\,b_j$] at  (3,1) {};
\node (b) [b,label=below:$\strut b$] at  (3,0) {};
\draw (a)--(ai) (b)--(bj) (a)--(b) (a)--(bj) (ai)--(ak);
\end{scope}
\end{tikzpicture}
\end{center}\vspace{1ex}
\caption{: The \ds\ in Lemma~\ref{small:triangles}}\label{fig:smalltriangles}
\end{figure}
Note that the edges adjacent to $b$ are unaffected by the \dsp.
\end{proof}

\begin{lemma} \label{small:fragments}
If $d\geq3$ and $d+1<n<2(d+1)$ then $\Mnd$ is connected.
Equivalently, the \ds\ chain is irreducible on the set of fragments with a given
vertex set.
\end{lemma}

\begin{proof}
It is known that the \emph{flip} chain~\cite{CDGH,MS} is irreducible on all $d$-regular connected graphs.  Furthermore,
Lemma~\ref{small:diameter} proves that all fragments are connected.
We will show that, if $d+1 < n<2(d+1)$ then a flip can be performed using at most three $\Delta$-switches. That is, if $G$ is a fragment and $G'$ is obtained from $G$ by a flip,
then there is a sequence of at most three \ds es which takes $G$ to $G'$.
The lemma then follows immediately.

We must consider the two possible types of flip, as shown in Figure~\ref{fig:fliptypes}. Here solid lines represent edges of $G$ and dashed lines represent the two non-edges to be inserted by the \ds.
We will always assume we wish to perform the flip which deletes the edges $v_1v_2$, $v_3v_4$  and inserts the edges $v_1v_3$, $v_2v_4$, and the edge $v_1v_4$ is present.
\begin{figure}
\begin{center}
\begin{tikzpicture}[scale=0.74,inner sep=0pt]
\begin{scope}[shift={(0,0)}]
\node (v1) [b,label=above left:$v_1$] at (0,2) {};
\node (v2) [b,label=above right:$v_2$] at (2,2) {} ;
\node (v3) [b,label=below left:$v_3$] at (0,0)  {};
\node (v4) [b,label=below right:$v_4$] at (2,0)  {};
\draw [thick,-] (v1) -- (v2) (v2) -- (v3) (v3) -- (v4) (v1) -- (v4);
\draw [thick,dashed] (v1) -- (v3) (v2) -- (v4) ;
\node at (1,-1) {(a)};
\end{scope}
\begin{scope}[shift={(7,0)}]
\node (v1) [b,label=above left:$v_1$] at (0,2) {};
\node (v2) [b,label=above right:$v_2$] at (2,2) {} ;
\node (v3) [b,label=below left:$v_3$] at (0,0)  {};
\node (v4) [b,label=below right:$v_4$] at (2,0)  {};
\draw [thick,-] (v1) -- (v2) (v3) -- (v4) (v1) -- (v4) ;
\draw [thick,dashed] (v1) -- (v3) (v2) -- (v4) ;
\node at (1,-1) {(b)};
\end{scope}
\end{tikzpicture}
\end{center}\vspace{1ex}
\caption{: The two types of flip}\label{fig:fliptypes}
\end{figure}
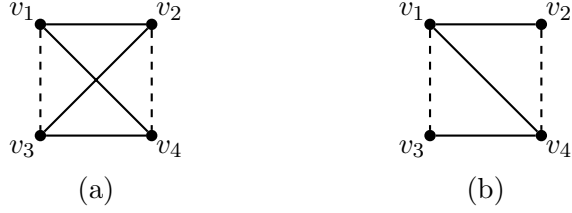
Note first that if $n<2d$ then there is a triangle on the vertices $u,v_1,v_2$ by Lemma~\ref{small:structure}. So we may use a single \dsm\ to perform a flip of either type, as shown at the top of Figure~\ref{fig:flipswitch1}.

For the remainder of the proof, suppose that $n\in \{2d, 2d+1\}$.
Let $S=\{v_i: i\in[4]\}$ be the vertices of the flip, and $N_i=\Nb(v_i)\setminus S$ ($i\in[4]$) be the neighbours of $v_i$ outside $S$. Note that either $|N_i|= d-1\geq 2$ or $|N_i|= d-2\geq 1$ ($i\in[4]$), since $d\geq 3$.
\begin{enumerate}
\item[(a)] First assume that $|N_i|= d-2$ for all $i\in[4]$.
If $N_1$, $N_2$, $N_3$ are pairwise disjoint then
\[ |N_1|+|N_2|+|N_3|=3(d-2)>|V\sm S|=n-4.\]
But this contradicts the fact that $n < 3d-2$
when $d\geq 4$ and $n\in \{2d,2d+1\}$, or when $d=3$ and $n=2d$. Hence
at least one of $N_1\cap N_2, N_1\cap N_3, N_2\cap N_3$ must be nonempty.

    If $N_1\cap N_2\neq\es$ then we can use a single \dsm\ to perform the flip,
    while if $N_1\cap N_3\neq\es$ then we can use a single \dsp. See Figure~\ref{fig:flipswitch1}.
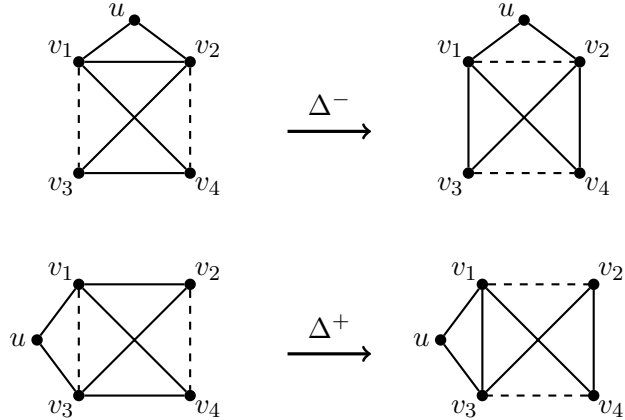
\begin{figure}[H]
\begin{center}
\begin{tikzpicture}[scale=0.74,inner sep=0pt]
\begin{scope}[shift={(0,0)}]
\node (u) [b,label= above left:$u\,$] at (1,2.75) {};
\node (v1) [b,label=above left:$v_1$] at (0,2) {};
\node (v2) [b,label=above right:$v_2$] at (2,2) {} ;
\node (v3) [b,label=below left:$v_3$] at (0,0)  {};
\node (v4) [b,label=below right:$v_4$] at (2,0)  {};
\draw [thick,-] (v2)--(v1) -- (v4) (v2) -- (v3)--(v4) (v1) -- (u) -- (v2) ;
\draw [thick,dashed] (v1) -- (v3) (v2) -- (v4)  ;
\node at (4.5,1.25) {\Dm} ; 
\draw  (3.75,0.75)edge[line width=1.2pt, ->](5.25,0.75);
\end{scope}
\begin{scope}[shift={(7,0)}]
\node (u) [b,label=above left:$u\,$] at (1,2.75) {};
\node (v1) [b,label=above left:$v_1$] at (0,2) {};
\node (v2) [b,label=above right:$v_2$] at (2,2) {} ;
\node (v3) [b,label=below left:$v_3$] at (0,0)  {};
\node (v4) [b,label=below right:$v_4$] at (2,0)  {};
\draw [thick,-] (v1) -- (v3)--(v2) (v4)--(v1)--(u)--(v2)--(v4);
\draw [thick,dashed] (v1) -- (v2) (v3) -- (v4)  ;
\end{scope}
\begin{scope}[shift={(0,-4)}]
\node (u) [b,label=left:$u\,$] at (-0.75,1) {};
\node (v1) [b,label=above left:$v_1$] at (0,2) {};
\node (v2) [b,label=above right:$v_2$] at (2,2) {} ;
\node (v3) [b,label=below left:$v_3$] at (0,0)  {};
\node (v4) [b,label=below right:$v_4$] at (2,0)  {};
\draw [thick,-] (v2)--(v1) -- (v4) (v2) -- (v3) (v1) -- (u) -- (v3)--(v4);
\draw [thick,dashed] (v1) -- (v3) (v2) -- (v4)  ;
\node at (4.5,1.25) {\Dp} ; 
\draw  (3.75,0.75)edge[line width=1.2pt, ->](5.25,0.75);
\end{scope}
\begin{scope}[shift={(7.25,-4)}]
\node (u) [b,label=left:$u\,$] at (-0.75,1) {};
\node (v1) [b,label=above left:$v_1$] at (0,2) {};
\node (v2) [b,label=above right:$v_2$] at (2,2) {} ;
\node (v3) [b,label=below left:$v_3$] at (0,0)  {};
\node (v4) [b,label=below right:$v_4$] at (2,0)  {};
\draw [thick,-] (v1)--(v4) (v3)--(v2) (v1) -- (v3) (v2) -- (v4) (v1) -- (u) -- (v3);
\draw [thick,dashed] (v1)--(v2) (v4)--(v3);
\end{scope}
\end{tikzpicture}
\end{center}\vspace{1ex}
\caption{: Two flips as \ds es}\label{fig:flipswitch1}
\end{figure}
If $N_2\cap N_3\neq\es$ then we must use a \dsm\ followed by a \dsp, to ensure the correct edges are flipped. See Figure~\ref{fig:flipswitch2}.
	    (The edges in each \ds\ can be deduced
	    by the non-edges shown in the figure, and the \ds\ type.)
\begin{figure}[H]
\begin{center}
\begin{tikzpicture}[scale=0.74,inner sep=0pt]
\node (u) [b,label=above left:${}_{}u\,$] at (1,2.75) {};
\node (v1) [b,label=below left:$v_1$] at (0,0) {};
\node (v2) [b,label=above left:$v_2$] at (0,2) {} ;
\node (v3) [b,label=above right:$v_3$] at (2,2)  {};
\node (v4) [b,label=below right:$v_4$] at (2,0)  {};
\draw [thick,-] (v1) -- (v2) (v3) -- (v4) (v2) -- (u) -- (v3) (v1)-- (v4) (v2)--(v3);
\draw [thick,dashed] (v1) -- (v3) (v2) -- (v4) ;
\node at (4.5,1.25) {\Dm} ; 
\draw  (3.75,0.75)edge[line width=1.2pt, ->](5.25,0.75);
\begin{scope}[shift={(7,0)}]
\node (u) [b,label=above left:${}_{}u\,$] at (1,2.75) {};
\node (v1) [b,label=below left:$v_1$] at (0,0) {};
\node (v2) [b,label=above left:$v_2$] at (0,2) {} ;
\node (v3) [b,label=above right:$v_3$] at (2,2)  {};
\node (v4) [b,label=below right:$v_4$] at (2,0)  {};
\draw [thick,-] (v1) -- (v3) (v2) -- (v4) (v2) -- (u) -- (v3) (v3)-- (v4) (v2)--(v1);
\draw [thick,dashed] (v2) -- (v3) (v1) -- (v4) ;
\node at (4.5,1.25) {\Dp} ; 
\draw  (3.75,0.75)edge[line width=1.2pt, ->](5.25,0.75);
\end{scope}
\begin{scope}[shift={(14,0)}]
\node (u) [b,label=above left:${}_{}u\,$] at (1,2.75) {};
\node (v1) [b,label=below left:$v_1$] at (0,0) {};
\node (v2) [b,label=above left:$v_2$] at (0,2) {} ;
\node (v3) [b,label=above right:$v_3$] at (2,2)  {};
\node (v4) [b,label=below right:$v_4$] at (2,0)  {};
\draw [thick,-] (v1) -- (v4)--(v2) -- (v3)--(v1) (v2) -- (u) -- (v3);
\draw [thick,dashed] (v2) -- (v1) (v3) -- (v4) ;
\end{scope}
\end{tikzpicture}
\end{center}\vspace{1ex}
\caption{: A flip as two \ds es}\label{fig:flipswitch2}
\end{figure}
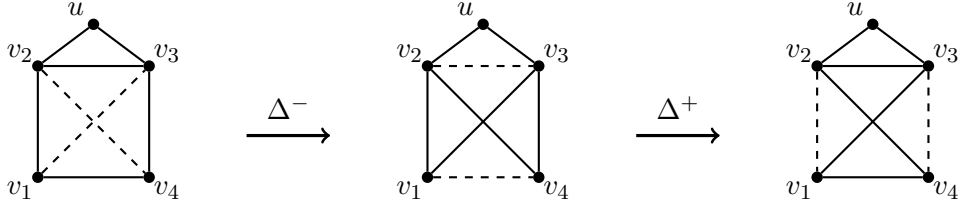
\item[(b)] Now assume that $|N_1|= |N_4|= d-2$ and $|N_2|= |N_3|= d-1$.
If $N_1\cap N_2\neq\es$ then we can perform a \dsm, while if $N_1\cap N_3\neq\es$ then we can perform a \dsp, as for (a) above.  The situation when $N_2\cap N_4\neq \es$
is symmetric with $N_1\cap N_3\neq \es$, so we may assume that $N_1$ is disjoint from
$N_2\cup N_3$, and $N_4$ is disjoint from $N_2$.  It follows that
\[ n\geq |N_1|+|N_2|+|S|=(d-2)+(d-1)+4=2d+1,\]
with equality only if $N_1=N_4$ and $N_2=N_3$.

Thus we may assume that $n=2d+1$ and write $V$ as the disjoint union
$V = N_1\cup N_2\cup S$, with $N_1=N_4$ and  $N_2=N_3$.
Each vertex in $N_1$ is adjacent to both $v_1$ and $v_4$, and hence has $d-2$ other neighbours. Similarly, each vertex in $N_2$ is adjacent to $v_2$ and $v_3$, and has $d-2$ other neighbours. Since $V=N_1\cup N_2\cup S$, all these edges lie in $G[N_1\cup N_2]$.
A total of $(d-1)(d-2)$ of these edges are incident on $N_2$, but only $(d-2)^2$
are incident on $N_1$. Therefore $(d-1)(d-2)-(d-2)^2=d-2\geq2$ edges must lie in $G[N_2]$. Thus $G$ contains the configuration shown in Figure~\ref{fig:flipswitch3} as a subgraph, and we show how to perform the required flip using three \ds es.
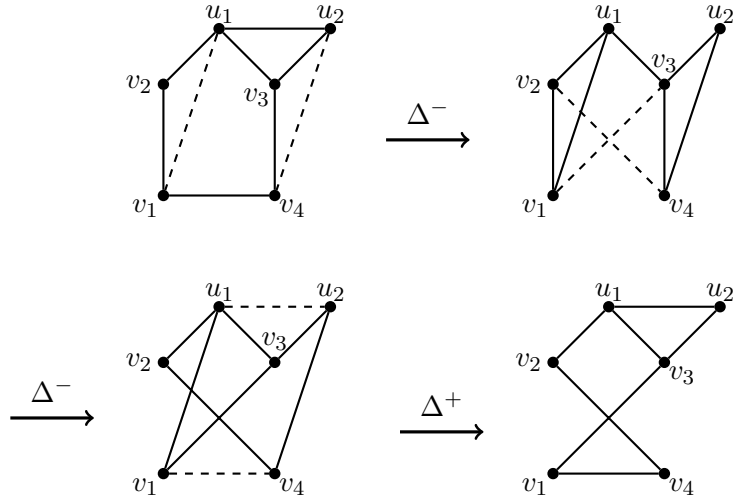
\begin{figure}[H]
\begin{center}
\begin{tikzpicture}[scale=0.74,inner sep=0pt]
\begin{scope}[shift={(0,0)}]
\node (u1) [b,label=above:${}_{}u_1$] at (1,3) {};
\node (u2) [b,label=above:${}_{}u_2$] at (3,3) {};
\node (v1) [b,label=below left:$v_1$] at (0,0) {};
\node (v2) [b,label=left:$v_2\,$] at (0,2) {} ;
\node (v3) [b,label=below left:$v_3$] at (2,2)  {};
\node (v4) [b,label=below right:$v_4$] at (2,0)  {};
\draw [thick,-] (v2)--(v1)--(v4) -- (v3) (v2) -- (u1) -- (v3) (u2)--(v3) (u1)--(u2);
\draw [thick,dashed] (u1)--(v1) (u2)--(v4) ;
\node at (4.75,1.5) {\Dm} ; 
\draw  (4.00,1.0)edge[line width=1.2pt, ->](5.5,1.0);
\end{scope}
\begin{scope}[shift={(7,0)}]
\node (u1) [b,label=above:${}_{}u_1$] at (1,3) {};
\node (u2) [b,label=above:${}_{}u_2$] at (3,3) {};
\node (v1) [b,label=below left:$v_1$] at (0,0) {};
\node (v2) [b,label=left:$v_2\,$] at (0,2) {} ;
\node (v3) [b,label=above:$\strut v_3$] at (2,2)  {};
\node (v4) [b,label=below right:$v_4$] at (2,0)  {};
\draw [thick,-] (v2)--(v1) (v4) -- (v3) (v2) -- (u1) -- (v3) (u2)--(v3) (u1)--(v1) (u2)--(v4);
\draw [thick,dashed] (v1)--(v3) (v2)--(v4) ;
\end{scope}
\begin{scope}[shift={(0,-5)}]
\node at (-2,1.5) {\Dm} ; 
\draw  (-2.75,1.0)edge[line width=1.2pt, ->](-1.25,1.0);
\node (u1) [b,label=above:${}_{}u_1$] at (1,3) {};
\node (u2) [b,label=above:${}_{}u_2$] at (3,3) {};
\node (v1) [b,label=below left:$v_1$] at (0,0) {};
\node (v2) [b,label=left:$v_2\,$] at (0,2) {} ;
\node (v3) [b,label=above:$\strut v_3$] at (2,2)  {};
\node (v4) [b,label=below right:$v_4$] at (2,0)  {};
\draw [thick,-] (v1)--(v3) (v2)--(v4) (v2) -- (u1) -- (v3) (u2)--(v3) (u1)--(v1) (u2)--(v4);
\draw [thick,dashed] (u1)--(u2) (v1)--(v4);
\end{scope}
\begin{scope}[shift={(7,-5)}]
\node at (-2,1.25) {\Dp} ; 
\draw  (-2.75,0.75)edge[line width=1.2pt, ->](-1.25,0.75);
\node (u1) [b,label=above:${}_{}u_1$] at (1,3) {};
\node (u2) [b,label=above:${}_{}u_2$] at (3,3) {};
\node (v1) [b,label=below left:$v_1$] at (0,0) {};
\node (v2) [b,label=left:$v_2\,$] at (0,2) {} ;
\node (v3) [b,label=below right:$v_3$] at (2,2)  {};
\node (v4) [b,label=below right:$v_4$] at (2,0)  {};
\draw [thick,-] (v1)--(v3) (v2)--(v4) (v2) -- (u1) -- (v3) (u2)--(v3) (u1)--(u2) (v1)--(v4) ;
\draw [thick,dashed];
\end{scope}
\end{tikzpicture}
\end{center}\vspace{1ex}
\caption{: A flip as three \ds es when $n=2d+1$}\label{fig:flipswitch3}
\end{figure}
\end{enumerate}
Here $u_1,u_2$ are any distinct elements of $N_2$ such that $u_1u_2\in E$.
Since $N_2=N_3$, both $u_1$ and $u_2$ are also neighbours of $v_3$.
Furthermore,
$u_1v_1$, $u_2v_4$ are both non-edges since $N_1\cap N_3 = N_4\cap N_3=\es$.
\end{proof}

\section{Creating a clique component containing a given vertex}\label{sec:split}

In this section we will prove the following.

\begin{theorem}\label{thm:split}
Suppose that $d \ge 3$ and $n \geq 2(d+1)$.  Given any $G\in \Gnd$ and any vertex
$v$ of $G$, let $S = \Nb_G[v]$ be the closed neighbourhood of $v$ in $G$.
Then there is a sequence of \ds es which ends in a graph $G'$  which has a
clique component on the vertex set $S$.
\end{theorem}

Note that $\Nb_{G'}[v] = \Nb_G[v]$, that is, the closed neighbourhood of $v$
is preserved by this process. This property will be used in Section~\ref{sec:sorting}.

\subsection{Proof strategy}\label{subsec:strategy}
We will consider $G$ as being layered from $v$, as defined in Section~\ref{sec:defs} above.
Let $C$ be the component of $G$ which contains $v$.
We prove that, provided $|C| \ge 2(d+1)$, there is a sequence of \ds es such that $V_1$ remains unchanged, but $|E_1|$ increases monotonically.
We repeat the following steps to add edges to $E_1$ until $G[V_0\cup V_1]\cong \Kd$.

\begin{enumerate}
\item \label{0} If, before any step below, the component $C$ containing $v$ is a fragment
(that is, if $d+1 < |C| < 2(d+1)$), use a \dsm\ to increase the size of $C$ to at least $2(d+1)$ without removing any edge in $E_1$. This can be done in such a way that $G_2$ now contains at least one non-edge, as we will prove
in Lemma~\ref{lem:biggerC}.

\item \label{1a}
While there is a vertex $u\in V_1$ which is not adjacent to any vertex in $V_1$
(that is, with $\dg(u)=0$)
make a \dsp\ to introduce an edge incident with $u$ in $G_1$.
That this is always possible will be proved in Lemma~\ref{lem:P0}.

After repeating this as many times as necessary, every vertex in $V_1$ will have an incident edge in $G_1$. Thus, every vertex $u\in V_1$ with a neighbour in $V_2$ will have $1 \le \dg(u) \le d-2$.

\item \label{1}
If $V_2=\emptyset$, stop and return the current graph as $G'$.
Otherwise,
while there is a non-edge $ab$ in $G_2$, insert edges into $E_1$ as follows:
\begin{enumerate}
\item Suppose that there is a unique $x \in V_1$ such that $a,b\in V_2$ are in $\Out(x)$ only. Thus $\id(a)=\id(b)=1$.
 Use  Lemma~\ref{lem:P1} to  make a \dsp\ which replaces edge $xb$ with $yb$ for some $y \in V_1$, $y \ne x$, thus giving a pair $x,y$ below non-edge $ab$.

\item Suppose that some pair $x,y$ below $ab$ is a non-edge of $G_1$. Use a \dsp\ at $v$ to switch $xa,yb$ to $xy,ab$; thus increasing the number of edges in $G_1$, as in Lemma~\ref{lem:P2}.

\item Now suppose that every pair $x, y$  below $ab$ is an edge $xy$ of $G_1$.
Choose one such pair and use the \ds\ at $v$ of Lemma~\ref{lem:P3} to make $xy$
a non-edge.  Then use a \ds\ at $v$ to switch $xa,yb$ to $xy,ab$.
\end{enumerate}

\item \label{2}
If $\ell=|V_2| \ge d+1$ then there are necessarily non-edges in $G_2$.
If  $\ell=d$ and $|C|  = 2(d+1)$,  then $V_3=\{u\}$, for some $u$, and $V_2=\Nb(u)$.
Again there are necessarily non-edges in $G_2$.
In either case go back to Step~\ref{1} above.

\item
If we reach here then $\ell \le d$ and $V_2$ is a complete graph $K_\ell$. If $V_3=\es$ then $|C| =1+d+\ell \le 2d+1$, a contradiction. Thus $V_3\ne \es$. The case $|C|=2(d+1)$ was covered in Step~\ref{2}, so we assume that $|C| > 2(d+1)$. Carry out the steps
in Lemmas~\ref{L1}--\ref{L3} to insert a non-edge into $G_2$, and go  to Step~\ref{1} above.
\end{enumerate}

\subsection{Increasing the size of $C$}\label{subsec:increase}

\begin{lemma}\label{lem:biggerC}
Suppose that $d\geq3$ and $n > 2(d+1)$.
If vertex $v$ is in a fragment $C$ then there is a \dsm\ to increase the size of
$C$ to at least $2d+3$ without changing the edges of $G_1$. After this switch, $V_2$ will contain a non-edge.
\end{lemma}
\begin{proof}
As $n \ge 2(d+1)$ the graph $G$ has another component $C'$ which must contain at
least $d+1$ vertices, with $v\notin C'$.

Let $ab$ be an edge with $a\in V_2$, $b\in V_1$. Then, using Lemma~\ref{small:triangles},
after at most one \ds\ there is a triangle incident with $a$.

Since $V_3=\es$, by Lemma~\ref{small:diameter}, there is a triangle on the vertices
$x,a,z$ with an edge $xa$ in $G_2$, or with edges $xa$, $za$ such that $x,z\in V_1$.
Choose any edge $x'y'\in G[C']$, and perform the \dsm\ on the triangle $x,a,z$ and edge $x'y'$ which removes the edges $xa$ and $x'y'$ and inserts the edges $xx',ay'$. If previously $xa\in G_2$, then after the \ds, $a$ remains in $G_2$ and $xa$ is now a non-edge in $G_2$. If $x,z\in V_1$ then we still have $z\in V_1$ with $az$ an edge, so $a$ remains in $V_2$. Also $x'\in V_2$, since it is now adjacent to $x\in V_1$. So $ax'$ becomes a non-edge in $G_2$. Note that these (at most two) \ds es do not change $G_1$.

After this, we will have $C\gets C\cup C'$, so $|C|\gets |C|+|C'| > 2d+2$.
\end{proof}
This procedure does not increase $|E_1|$ but, as we show next, the non-edge in $G_2$ allows us to increase $|E_1|$ with at most two further \ds es.
Thus the process outlined in Section~\ref{subsec:strategy} must terminate in a finite
number of steps. 

\subsection{$G_2$ has a non-edge}\label{sec:non-edge}
Recall that a \ds\ at $v$,
as in Figure~\ref{fig:fig1},
can be regarded as a switch in $G\setminus \{v\}$.  We take this view below.

\begin{lemma}\label{lem:P0}
Suppose that $d\geq 3$.
Let $C$ be the component of $G$ which contains $v$, and suppose that
$|C|\geq 2(d+1)$. If $u \in V_1$ has $\dg(u)=0$ then we can use at most two
\ds es to
insert an edge in $V_1$ at $u$, without altering any other edge in $E_1$.
\end{lemma}

\begin{proof}
As $\dg(u)=0$, we have $\od(u)=d-1$. Thus $\ell \ge d-1$.
Let $B=V_2\sm\Out(u)$ and let $h=\ell-(d-1)$ be the number of vertices in $B$.

Suppose first that $h>0$ (that is, $\ell\geq d$). There are 3 subcases.

(i) Some vertex $x \in \Out(u)$ has $\id(x) \ge 2$, so $x \in \Out(w)$, $w \ne u$. There are at most $d-2$ edges incident with $x$ in $G_2$, so there is a non-edge $xz$ in $G_2$. Let $y \in \In(z)$. If $y=u$ then switch $uz,wx$ to $uw,xz$, else switch $ux,yz$ to $uy,xz$.

(ii) Some vertex $x \in B$ has $\id(x) \ge 2$. As $|\Out(u)|=d-1$ there is a non-edge $xz$, $z \in \Out(u)$. Since $x \in B$ there is some $w\neq u$ with  $x \in \Out(w)$. Switch $uz,wx$ to $uw,zx$.

(iii) All vertices $x$ in $V_2$ have $\id(x)=1$. If there is a non-edge from
$\Out(u)$ to $B$, say $x\in \Out(u)$ and $z\in B$, then let $w\in V_1$ such that
$z\in \Out(w)$.  By assumption, $w\neq u$ and we can switch $ux$, $wz$ to $uw$, $xz$.
Now suppose that there is no non-edge from $\Out(u)$ to $B$.
Recall that $|\Out(u)|=d-1\geq 2$.
Each vertex in $V_1 \setminus \{u\}$ has at least one neighbour in $V_2$
(even if $G[V_1 \setminus \{u\}]$ is complete), and all of these neighbours
are distinct by assumption.  Hence $|B|\geq d-1\geq 2$.  
Choose $x,x' \in \Out(u)$ and $z,z'\in B$, and
let $w,w'\in V_1$ be such that $z \in \Out(w)$ and $z' \in \Out(w')$. By assumption,
$w\neq w'$ so we can use the \dsp\ $zxux'z'$ to remove $xz$, then switch $ux,wz$
to $uw$, $xz$.

Thus we can assume that $h=0$ and  $\ell=d-1$.
In this case $B=\es$.

If some vertex in $x\in V_2$ has $\id(x)\geq 3$ then $\dg(x)\leq d-3$,
and so $x$ has a non-edge in $G_2$.
Then we can proceed as in (i) above.
Otherwise, the $d-1$ vertices in $V_1 \setminus \{ u\}$ each have at least
one neighbour in $V_2$, and these must be distinct since no $x\in V_2$ has
$\id(x)\geq 3$. Hence
there are exactly $2(d-1)$ edges between $V_1$ and $V_2$
and all vertices of $V_2$ have in-degree 2.
Again, if there is a non-edge $xz\in G_2$ then we can switch $ux$, $wz$
to $uw$, $xz$, where $z\in \Out(w)$ and $w\in V_1\setminus \{u\}$.
Otherwise, $G_2$ is isomorphic to $K_{d-1}$, and any $x\in V_2$ has
$\od(x)=d-2-(d-2)=0$.
Hence $V_3=\es$ and $|C|=1+d+(d-1)=2d$. So $C$ is a fragment,
contradicting $|C|\geq 2(d+1)$. Thus this case cannot occur.
\end{proof}

Let $ab$ be a non-edge of $V_2$ above $x,y \in V_1$. We will show that we can rearrange the edges of $G_1$ as necessary to enable a \ds\ $xavyb$, replacing $xa,yb$ with $xy,ab$, inserting an edge $xy$ into $E_1$. Lemma~\ref{lem:P1} deals with the case where $ab$ lies uniquely within $\Out(u)$ for some $u \in V_1$. Lemmas~\ref{lem:P2} and~\ref{lem:P3} interchange edges and non-edges in $G_1$ if necessary. First, we show that we can assume that every non-adjacent pair $a,b\in V_2$ is above some pair in $V_1$.

\begin{lemma}\label{lem:P1}
Let $d \ge 3$ and $\dg(u)\geq 1$ for all $u\in V_1$.  Let $a,b\in V_2$ be a pair
of distinct non-adjacent vertices such that
$\In(a)=\In(b)=\{x\}$ for some $x\in V_1$.
Then there is a \ds\ at $v$ to move $b$ to $\Out(y)$ for some $y\in V_1$, without altering $E_2$,
so that $a,b$ is above $x,y$.
 \end{lemma}

\begin{proof}
In Figure~\ref{fig:figP6n}, $xb$ is an edge and so $\dg(x)\leq d-2$. Hence there is a non-edge $xw$ for some $w \in V_1$. As $\dg(w) \geq 1$ there is some $y \in V_1$ such that $wy$ is an edge. Clearly $y\neq x$. Note that $yb$ is a non-edge because $\id(b)=1$. (Pairs not shown as an edge or non-edge can be either.) Now switch $xb, wy$ to $xw,by$.
\end{proof}
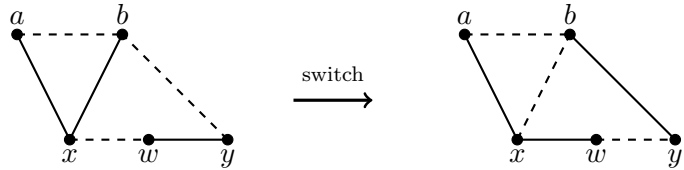
\begin{figure}[H]
\begin{center}
\begin{tikzpicture}[scale=0.7]
\draw[thick, -]  (1,0)--(2,2) (1,0)--(0,2);
\draw[thick, dashed] (0,2)--(2,2);
\draw[thick, -] (2.5,0)--(4,0);
\draw[thick, dashed] (1,0)--(2.5,0) (4,0)--(2,2);
\node [below] at (1,0)  {$x$};
\node [below] at (2.5,0)  {$w$};
\node [below] at (4,0)  {$y$};
\node [above] at (0,2)  {$a$};
\node [above] at (2,2)  {$b$};
\draw [fill] (1,0) circle (0.1);
\draw [fill] (4,0) circle (0.1);
\draw [fill] (2.5,0) circle (0.1);
\draw [fill] (0,2) circle (0.1);
\draw [fill] (2,2) circle (0.1);
\node at (6.0,1.25) {\scriptsize  switch} ; 
\draw  (5.25,0.75)edge[line width=1.2pt, ->](6.75,0.75);
\begin{scope}[shift={(8.5,0)}]
\draw[thick, dashed] (0,2)--(2,2);
\draw[thick, dashed]  (1,0)--(2,2);
\draw[thick, dashed] (2.5,0)--(4,0);
\draw[thick, -] (1,0)--(2.5,0) (4,0)--(2,2) (1,0)--(0,2);
\node [below] at (1,0)  {$x$};
\node [below] at (2.5,0)  {$w$};
\node [below] at (4,0)  {$y$};
\node [above ] at (0,2)  {$a$};
\node [above ] at (2,2)  {$b$};
\draw [fill] (1,0) circle (0.1);
\draw [fill] (4,0) circle (0.1);
\draw [fill] (2.5,0) circle (0.1);
\draw [fill] (0,2) circle (0.1);
\draw [fill] (2,2) circle (0.1);
\end{scope}
\end{tikzpicture}
\end{center}
\caption{: The switch in Lemma~\ref{lem:P1}, which changes $\Nb(b)\cap V_1$}
\label{fig:figP6n}
\end{figure}

\begin{lemma} \label{lem:P2}
Let $ab$ be a non-edge of $G_2$, above a non-edge $xy$ in $G_1$. Then there is a \dsp\ to put $xy\in E_1$ without altering any other edges of $E_1$.
\end{lemma}
\begin{proof}
Clearly $axvyb$ is the required \dsp.
\end{proof}
\begin{figure}[H]
\begin{center}
\begin{tikzpicture}[scale=0.7]
\draw[thick, -]  (0,0)--(0,2) (2,0)--(2,2);
\draw[thick, dashed] (0,0)--(2,0) (0,2)--(2,2) ;
\node [b,label=below:$x$] at (0,0)  {};
\node [b,label=below:$y$] at (2,0)  {};
\node [b,label=above:$a$] at (0,2)  {};
\node [b,label=above:$b$] at (2,2)  {};
\node at (5,1.25) {\scriptsize switch} ; 
\draw  (4.25,0.75)edge[line width=1.2pt, ->](5.75,0.75);
\begin{scope}[shift={(8,0)}]
\draw[thick, dashed]  (0,0)--(0,2) (2,0)--(2,2);
\draw[thick, -] (0,0)--(2,0) (0,2)--(2,2) ;
\node [b,label=below:$x$] at (0,0)  {};
\node [b,label=below:$y$] at (2,0)  {};
\node [b,label=above:$a$] at (0,2)  {};
\node [b,label=above:$b$] at (2,2)  {};
\end{scope}
\end{tikzpicture}
\end{center}
\caption{: The switch in Lemma~\ref{lem:P2} which inserts $xy$ into $E_1$. }
\label{fig:figP7n}
\end{figure}
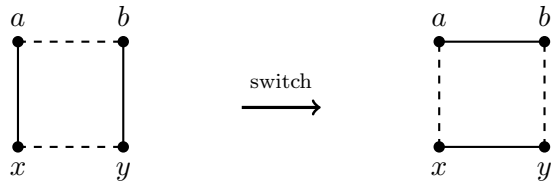

\begin{lemma} \label{lem:P3}
Let $d \ge 3$, and let $xy\in E_1$ be such that $\od(y)\geq1$. Then there is a \ds\ which makes $xy$ a non-edge,
without changing $E_2$ or decreasing $|E_1|$.
 \end{lemma}

\begin{proof}
Consider the graph $H=G[V_1\cup V_2]$. Note that all $u\in V_1$ have $\deg_H(u)=d-1$.
Since $\dg(y)=d-1-\od(y)\leq d-2$,
there exists $w\in V_1$ such that $yw\notin E_1$, as in Figure~\ref{fig:figP8n}.
First suppose that $xw\notin E_1$. Let $W=\Nb(w)\sm \{v\}$ and $X=\Nb(x)\sm \{v,y\}$, so $|W|=d-1$, $|X|=d-2$. Thus there exists $z\in W\sm X$. So $wz\in E(H)$ and $xz\notin E(H)$, and there is a switch replacing $xy,wz$ by $yw,xz$. Now $|E_1|$ is unchanged but
$xy$ is a non-edge. If $z\in V_1$ then two edges in $G_1$ are added and two removed. If $z\in V_2$ then one edge of $G_1$ is added and one is removed. No edges of $G_2$ are changed, since only $z$ can be in $V_2$. Finally, if $xw\in E_1$, let $W=\Nb(w)\sm \{v,x\}$ and $X=\Nb(x)\sm \{v,w,y\}$. Then $|W|=d-2$ and $|X|=d-3$. The remainder of the argument proceeds as above.
\end{proof}

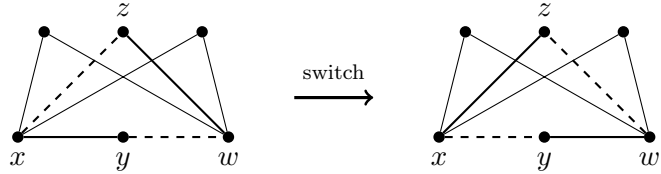
\begin{figure}[H]
\begin{center}
\begin{tikzpicture}[scale=0.7]
\begin{scope}[shift={(0,0)}]
\node (x) [b,label=below:$x$] at (0,0)  {};
\node (y) [b,label=below:$y$] at (2,0)  {};
\node (w) [b,label=below:$w$] at (4,0)  {};
\node (z) [b,label=above:$z$] at (2,2)  {};
\node (z1) [b] at (0.5,2)  {}; \node (z2) [b] at (3.5,2)  {};
\draw[thick] (x)--(y) (w)--(z);
\draw (x)--(z1)--(w) (x)--(z2)--(w);
\draw[thick,dashed] (y)--(w) (x)--(z);
\node at (6,1.25) {\scriptsize  switch} ; 
\draw  (5.25,0.75)edge[line width=1.2pt, ->](6.75,0.75);
\end{scope}
\begin{scope}[shift={(8,0)}]
\node (x) [b,label=below:$x$] at (0,0)  {};
\node (y) [b,label=below:$y$] at (2,0)  {};
\node (w) [b,label=below:$w$] at (4,0)  {};
\node (z) [b,label=above:$z$] at (2,2)  {};
\node (z1) [b] at (0.5,2)  {}; \node (z2) [b] at (3.5,2)  {};
\draw[thick] (y)--(w) (x)--(z);
\draw (x)--(z1)--(w) (x)--(z2)--(w);
\draw[thick,dashed](x)--(y) (w)--(z);
\end{scope}
\end{tikzpicture}
\end{center}
\caption{: The switch in Lemma~\ref{lem:P3}, making $xy$ a non-edge. }
\label{fig:figP8n}
\end{figure}
If $ab$ was a non-edge above the edge $xy$, then both $\od(x),\od(y)\geq 1$, so Lemma\ref{lem:P3} applies.
After performing the \ds\ from Lemma~\ref{lem:P3}, $ab$ will be above the non-edge $xy$. 
Then we can use Lemma~\ref{lem:P2} to re-insert $xy$.

\subsection{$G_2$  has no non-edges}\label{subsec:S2}

As usual, $C$ denotes the component of $G$ containing $v$. We assume that
$|C|\geq 2(d+1)$ and that $G_2$ is complete.
If  (i) $V_4 \ne \es$ and $\ell \ge 2$ or (ii) all vertices of $G_2$ have $d-\ell$ edges to $V_3$, then we can 
use \ds es to create a non-edge in $G_2$.
This is proved in Lemmas~\ref{L2} and~\ref{L3} respectively. If these conditions are not met, then Lemma~\ref{L1} describes a procedure which can be repeated until all vertices of $V_2$ have in-degree one. As a consequence $\ell \ge 2$, and all $u\in V_2$ have $\od(u)=d-(\ell-1)-1=d-\ell$, thus satisfying Lemma~\ref{L3}.

If $\ell \ge d+1$ then $G_2$ necessarily has a non-edge, so  assume $\ell \le d$.
Then $\id(u)\geq 1$, for any $u \in V_2$, and $\dg(u) = \ell-1$ as $G_2$
is complete. Thus
 $\od(u)\leq d-\ell$, so there are at most $\ell(d-\ell)$ edges from $V_2$ to $V_3$. If $\ell=d$ then $V_3=\es$ and so $|C|=2d+1$, a contradiction.
 Similarly, if $|V_3| =1$ then $\ell\geq d$ as $V_4=\emptyset$, and again we
 have a non-edge in $G_2$ or a contradiction. 
 If $|C|=2(d+1)$ then $V_2\cup V_3$ is a $d$-regular subgraph on $d+1$ vertices,
 which must be isomorphic to $\Kd$.  But this contradicts the fact that
 all vertices $u\in V_2$ have $\id(u)\geq 1$.
Hence we may assume that $1 \le \ell \le d-1$, $|C| > 2(d+1)$, and $|V_3| \ge 2$.

\begin{lemma}\label{L1}
Suppose that $d\geq 3$ and $|C| > 2(d+1)$.
Further suppose that $1\leq \ell \leq d-1$ and $|V_3|\geq 2$, with $G_2$ complete and
$V_4=\emptyset$.
If some $u \in V_2$ has $\id(u) \ge 2$ then
there is a \dsp\ which reduces $\id(u)$ by one and moves a vertex of
$V_3$ to $V_2$, without altering $E_1$.
\end{lemma}

\begin{proof}
Partition $V_2$ into $A,B$ where $A=\{w\in V_2:\od(w)=0\}$, and $B=V_2\sm A$.
Note that every $x\in V_3$ has a neighbour in $B$, so $1\leq |B|\leq \ell$.

First suppose that $u \in A$, and choose $x\in V_3$. Then $ux$ is a non-edge.
Furthermore, $x$ has a neighbour $y \in V_3$, since $\ell \le d-1$ (and using
the fact that $V_4=\emptyset$).
Let $w$ be any neighbour of $x$ in $B$, and let $z$ be a neighbour of $u$ in $V_1$.
As $ux, yz$ are non-edges, the path $zuwxy$ describes a \dsp\ at $w$.

Now suppose that $u \in B$. Let $uw$ be an edge  to $w \in V_3$.
As $u$ has $\od(u)\leq d-2-(\ell-1)=d-\ell-1$ neighbours in $V_3$ and $w$ has at least
$d-\ell$ neighbours in $V_3$, it follows that $w$ has a neighbour $x \in V_3$ with non-edge $ux$.
As $\id(u) \ge 2$ and $u\notin A$, we have $\ell-1=\dg(u)\leq d-3$, so $\ell \le d-2$. Thus $\dg(x)+\od(x)\geq d-(d-2)=2$,
so $x$ has another neighbour $y \in V_3$.
Let $uz$ be an edge to any neighbour $z$ of $x$ in $V_1$. Then $xu$ is a non-edge by construction, and $yz$ is a non-edge as $y,z$ belong to non-consecutive levels.
Hence $zuwxy$ is a \dsp\ at $w$.

In either case, after making the \dsp\ we have $u,y \in V_2$, $\id(y)=1$, and $\id(u)\gets \id(u)-1$.
\end{proof}
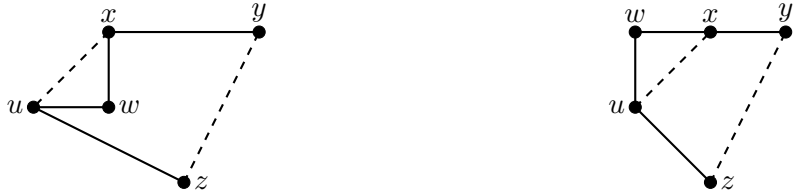
\begin{figure}[H]
\begin{center}
\begin{tikzpicture}[scale=1]
\draw[thick] (2,0)--(0,1)--(1,1)--(1,2)--(3,2);
\draw[thick, dashed] (1,2)--(0,1)  (2,0)--(3,2);
\draw [fill] (1,1) circle (0.08);\draw [fill] (0,1) circle (0.08);
\draw [fill] (2,0) circle (0.08);\draw [fill] (1,2) circle (0.08);
\draw [fill] (3,2) circle (0.08);
\node [right] at (2,0)  {$z$};
\node [right] at (1,1)  {$w$};
\node [above] at (1,2)  {$x$};
\node [above] at (3,2)  {$y$};
\node[left] at (0,1){$u$};

\begin{scope}[shift={(8,0)}]
\draw[thick] (1,0)--(0,1)--(0,2)--(1,2)--(2,2);
\draw[thick, dashed] (1,0)--(2,2) (0,1)--(1,2);
\draw [fill] (1,0) circle (0.08);\draw [fill] (0,1) circle (0.08);
\draw [fill] (0,2) circle (0.08);\draw [fill] (1,2) circle (0.08);
\draw [fill] (2,2) circle (0.08);
\node [right] at (1,0)  {$z$};
\node [left] at (0,1)  {$u$};
\node [above] at (0,2)  {$w$};
\node[above]at(1,2) {$x$};
\node[above]at(2,2) {$y$};
\end{scope}
\end{tikzpicture}
\end{center}
\caption{: Illustration of the edge structure for  the two cases in Lemma~\ref{L1}}
\label{fig:figL1}
\end{figure}

The above process can be repeated until there is a non-edge in $V_2$, in which case we proceed as in Section~\ref{sec:non-edge}, or all vertices on $V_2$ have in-degree one. In this case $\ell \ge 2$, because there must be at least two vertices in $V_1$ with out-degree at least one, or else $G[V_1]=K_d$ and we are done.

Thus we may now assume below that $G_2$ is complete with $2\leq \ell\leq d-1$,
all vertices $u\in V_2$ have $\id(u)=1$, and $V_3\ne \es$ (or else $C$ is a fragment).
Hence all $u\in V_2$ have $\od(u)=d-1-(\ell-1)=d-\ell$, and there are exactly $\ell(d-\ell)$ edges between $V_2$ and $V_3$.

\begin{lemma} \label{L2}
Suppose that $G_2$ is complete and $2\leq \ell\leq d-1$,
all vertices $u\in V_2$ have $\id(u)=1$, and $V_3\ne \es$.
If $V_4 \ne \es$ then we can apply a \dsp\ to create a non-edge in $G_2$ without altering
$E_1$.
\end{lemma}

\begin{proof}
If $G[V_4\cup V_5]$ contains an edge then there is a path $uwx$, where $u \in V_3$, $w \in V_4$ and $x \in V_4\cup V_5$.
If $\ell \ge 2$ then there is a path $uyz$ where $y,z \in V_2$. Difference in levels implies $wy, xz$ are non-edges. Then the \dsp\ $xwuyz$ makes $yz$ a non-edge of $G_2$, and does not alter $G_1$.

Otherwise $V_4$ is an independent set, $V_5=\es$ and hence $|V_3| \ge d$.
For a given $u\in V_4$, let $S=\Nb(u)\seq V_3$, so $|S| = d$. The total in-degree of $S$ is at most $\ell(d-\ell)$, so if every vertex $w$ of $S$ has $\id(w)=\ell$ then $|S| \le d-\ell$, a contradiction. Thus there is some $z \in V_2$ and $w\in S$ such that $zw$ is a non-edge. As $\od(z)\leq d-\ell \le d-2$, for some $x \in S$, $x\ne w$, there is an edge $xy$ to some $y \ne z$ in $V_2$. Also $zy\in E_2$, since $G_2$ is complete. Thus $wuxyz$ is a path with $wz, uy$ non-edges, and so the \dsp\ $wuxyz$ makes $yz$ a non-edge of $G_2$, and does not affect~$G_1$.
\end{proof}

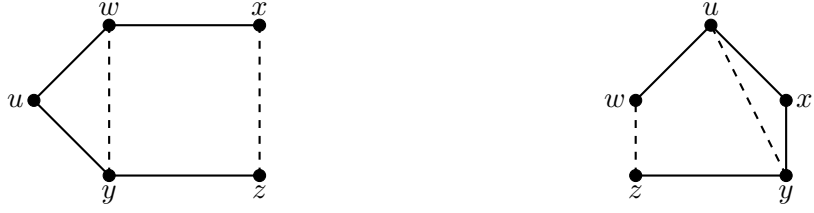
\begin{figure}[H]
\begin{center}
\begin{tikzpicture}[scale=1]
\draw[thick] (1,0)--(0,1)--(1,2)--(3,2) (1,0)--(3,0);
\draw[thick, dashed] (1,0)--(1,2)  (3,0)--(3,2);
\draw [fill] (1,0) circle (0.08);\draw [fill] (3,0) circle (0.08);
\draw [fill] (0,1) circle (0.08);\draw [fill] (1,2) circle (0.08);
\draw [fill] (3,2) circle (0.08);
\node [below] at (1,0)  {$y$};
\node [below] at (3,0)  {$z$};
\node [above] at (1,2)  {$w$};
\node [above] at (3,2)  {$x$};
\node[left] at (0,1){$u$};

\begin{scope}[shift={(8,0)}]
\draw[thick] (0,0)--(2,0)--(2,1)--(1,2)--(0,1);
\draw[thick, dashed] (0,0)--(0,1) (2,0)--(1,2);
\draw [fill] (0,0) circle (0.08);\draw [fill] (2,0) circle (0.08);
\draw [fill] (2,1) circle (0.08);\draw [fill] (1,2) circle (0.08);
\draw [fill] (0,1) circle (0.08);
\node [below] at (0,0)  {$z$};
\node [below] at (2,0)  {$y$};
\node [above] at (1,2)  {$u$};
\node[left]at(0,1) {$w$};
\node[right]at(2,1) {$x$};
\end{scope}
\end{tikzpicture}
\end{center}
\caption{: Illustration of the edge structure for  the two cases in Lemma~\ref{L2}}
\label{fig:figL2}
\end{figure}

\begin{lemma} \label{L3}
Let $d\geq 3$ and $|C|\geq 2(d+1)$.
Suppose that $G_2$ is complete,$2\leq \ell\leq d-1$,
and all $u\in V_2$ have $\id(u)=1$.
Further suppose that $V_3\neq \emptyset$ and $V_4=\es$.
Then there is a \dsm\ in $G[V_2\cup V_3]$ to remove an edge of $V_2$
without altering $E_1$.
\end{lemma}

\begin{proof}
It follows from the assumptions that there are exactly $\ell(d-\ell)$
edges from $V_2$ to $V_3$, as explained earlier.
Since $V_4=\es$ we have $d-\ell\leq \dg(u)\leq |V_3|-1$ for any $u\in V_3$,
so $|V_3|\geq d-\ell+1$. There are $\ell(d-\ell)$ edges from $V_2$ to $V_3$, so less than $d-\ell$ vertices $u\in V_3$  can have $\id(u)=\ell$.

If some $w \in V_3$ has $\id(w)\leq \ell-2$, then $\id(w)\geq 1$ implies $\ell \ge 3$.
Thus $\dg(w)\geq d-\ell+2$, but at most $d-\ell$ vertices in $\Nb'(w)$ can have in-degree $\ell$.
So $w$ has a neighbour $u$ in $V_3$ with $\id(u) \le \ell-1$. Thus, for some $y \in V_2$, $yu$ is a non-edge.
Also, $w$ has at least one non-neighbour $x\neq y$ in $V_2$, as $\id(w)\leq \ell-2$.
Let $z\in V_2$ be such that $z\neq x,y$. This vertex exists because $\ell\geq3$, and $G$ has a
triangle on the vertices $x,y,z$ because $G_2$ is complete. Now use a \dsm\ on remove $wu$ and the edge $xy$ of triangle $x,y,z$, and insert the edges $wx$, $uy$, as in Figure~\ref{fig:figL3}. This removes
one edge of $G_2$ without changing other edges of $G_2$ or altering $G_1$.

Finally suppose that all vertices of $ V_3$ have in-degree  $\ell-1$ or $\ell$.
Let $|V_3|=d-\ell+h$, where we know $h>0$.
Any $u \in V_2$ has $d-\ell+h-(d-\ell)=h$
non-neighbours in $V_3$. Since each vertex in $V_3$ has at most one non-neighbour in $V_2$,
there are exactly $h\ell$ vertices in $V_3$ with in-degree $\ell-1$, and hence $d-\ell+h-h\ell=d-\ell-h(\ell-1)$ vertices of in-degree $\ell$.

Let $w\in V_3$ have $\id(w)=\ell-1$, and thus exactly one non-edge to some $x \in V_2$.
So $\dg(w)=|\Nb'(w)|=d-\ell+1$.
In $\Nb'(w)$ there are at most $h$ 
vertices which may have a non-edge to $x$, and there are at most $d-\ell-h(\ell-1)$
vertices of in-degree $\ell$. Therefore, since $\ell \ge 2$, $w$ has at least
\[ d-\ell+1-h-(d-\ell-h(\ell-1))=h(\ell-2)+1 > 0 \]
neighbours $u$ of $w$ in $V_3$ with a (unique) non-edge to some $y \ne x$
in $V_2$. Thus, as before, there is a \dsm\ removing $wu$ and $xy$ and inserting edges
$wx$, $uy$, using any triangle $x,y,z$ in $G_2$, as in Figure~\ref{fig:figL3}.
\end{proof}
\begin{figure}[H]
\begin{center}
\begin{tikzpicture}[scale=1]
\draw[thick] (1,0)--(-1,0) (1,2)--(3,2) (1,0)--(3,0) ;
\draw[thick, dashed] (3,0)--(3,2) (1,0)--(1,2) ;
\draw [fill] (1,0) circle (0.08);\draw [fill] (3,0) circle (0.08);
\draw [fill] (-1,0) circle (0.08);\draw [fill] (1,2) circle (0.08);
\draw[thick] (-1,0)edge[bend right=25](3,0);
\draw [fill] (3,2) circle (0.08);
\node [above left] at (1,0)  {$x$};
\node [right] at (3,0)  {$y$};
\node [above] at (1,2)  {$w$};
\node [above] at (3,2)  {$u$};
\node [left] at (-1,0) {$z$};
\end{tikzpicture}
\end{center}
\caption{: Illustration of the \dsm\ in Lemma~\ref{L3}}
\label{fig:figL3}
\end{figure}
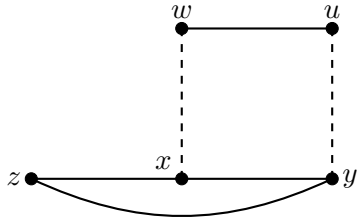
After these steps $G_2$ has a non-edge, and we can return to Section~\ref{sec:non-edge}.

\begin{remark}\label{rem:bound}
While our focus is not on the efficiency of the process described in Section~\ref{subsec:strategy},
we can bound the number of \ds es required to create a clique component. 
The steps in Lemmas~\emph{\ref{lem:P0}--\ref{lem:P3}} require $\Theta(1)$ \ds es for each edge inserted in $G_1$, so $\Theta(d^2)$ in total. The steps in Lemmas~\emph{\ref{L1}--\ref{L3}} also  require $O(1)$ \ds es, with the exception of Lemma~\emph{\ref{L1}}, which could possibly be executed $\Theta(d)$ times between edge insertions in $V_1$. Thus the total number of \ds es required is $\Omega(d^2)$ and $O(d^3)$. Note that this is independent of $n$, since at most five layers of $G$ are involved in the process.
\end{remark}

\section{Relabelling the vertices of clique components}\label{sec:sorting}

To complete the proof of Theorem~\ref{thm:irreducible}, we need to show that any graph $X=(V,E_X)\in\Gnd$ can be transformed to any other graph $Y=(V,E_Y)\in\Gnd$ with a sequence of \ds es. We will do this by induction on $n$. It is trivially true for $n=d+1$, since $\Gnd$ contains only one labelled graph, $\Kd$.
We know from Lemma~\ref{small:fragments} that $\Mnd$ is connected for $d+1<n<2(d+1)$. For $n\geq 2(d+1)$, we will assume inductively that $\Mnd[n',d]$ is connected
for all $n'<n$.

Choose any $v\in V$.  First, suppose that $\Nb_X(v)=\Nb_Y(v)$.
We know from Section~\ref{sec:split} that we can
perform a sequence of \ds es to transform $X$ into a graph which is a disjoint
union of
a clique component on the vertex set $\Nb_X[v]$ and a $d$-regular
graph $X'$ with $n-d-1$ vertices. 
Similarly, we can perform a sequence of \ds es to transform $Y$ into a
disjoint union of a clique component on the vertex set $\Nb_Y[v]$
and a $d$-regular graph $Y'$ with $n-d-1$ vertices.
Since $\Nb_X(v) = \Nb_Y(v)$, it follows that $X'$ and $Y'$ have the
same vertex set.

Hence, by induction, there is a sequence of \ds es that transforms $X'$ into $Y'$,
as required.

Now suppose that $\Nb_X(v)\neq \Nb_Y(v)$.  Using the above procedure,
we can assume that $\Nb_X[v]$ spans a copy of $K_{d+1}$, and similarly for $Y$.
We now show how to perform a sequence of switches, starting from $X$, to
ensure that the neighbourhood of $v$ matches $\Nb_Y(v)$.

Let $x\in \Nb_X(v)\sm \Nb_Y(v)$ and $y\in \Nb_X(v)\sm \Nb_Y(v)$. Note that $y$ exists because $|\Nb_X(v)|= |\Nb_Y(v)|$. Since $y\notin \Nb[v]$, it must be the case that $y$ is a vertex of $X'$. Therefore, let $yz$ be any edge of $X'$ incident on $y$, and let $w$ be any vertex of $\Nb(v)\sm\{x\}$, which exists since $d\geq3$. Note that $G$ has a triangle on
the vertices $v,w,x$, since $\Nb_X[v]$ spans a copy of $\Kd$.
Then we may perform a \dsm\ as shown in Figure~\ref{fig:sortingxy}.
\begin{figure}[H]
\begin{center}
\begin{tikzpicture}[scale=1.1]
\node (v) [b,label=below:$v$] at (0,0)  {};
\node (w) [b,label=above:$w$] at (-0.5,1) {};
\node (x) [b,label=above:$x$] at (0.5,1) {};
\node (z) [b,label=right:$z$] at (1.5,1.0) {};
\node (y) [b,label=right:$y$] at (1.5,0.0) {};
\draw (v)--(w)--(x)--(v) (y)--(z);
\draw[dashed] (x)--(z) (y)--(v) ;
\begin{scope}[shift={(7,0)}]
\node at (-2.5,0.85) {\Dm} ; 
\draw  (-3.0,0.4)edge[line width=1.2pt, ->](-2.0,0.4);
\node (v) [b,label=below:$v$] at (0,0)  {};
\node (w) [b,label=above:$w$] at (-0.5,1) {};
\node (x) [b,label=above:$x$] at (0.5,1) {};
\node (z) [b,label=right:$z$] at (1.5,1.0) {};
\node (y) [b,label=right:$y$] at (1.5,0.0) {};
\draw (y)--(v)--(w)--(x)--(z) ;
\end{scope}
\end{tikzpicture}
\end{center}\vspace{1ex}
  \caption{: Swapping $x,y$ in $\Nb_X(v)$}\label{fig:sortingxy}
\end{figure}
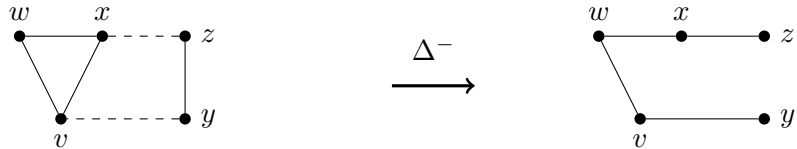

This creates a graph, which we rename as $X$, such that $\Nb_X(v)\gets \Nb_X(v)\sm \{x\}\cup \{y\}$. From this we can perform a sequence of \ds es to create a clique component
with vertex set $\Nb_X[v]$, using the method of Section~\ref{sec:split}.
After this iteration, 
again renaming the new graph as $X$, we find that $\Nb_X[v]$ spans a clique component
and
$|\Nb_X[v]\cap \Nb_Y[v]|$ has been increased by 1. After at most $d$ repetitions of this process, we have reached a new graph $X$ such that $\Nb_X(v)=\Nb_Y(v)$. Now we
may follow the argument given above for that case, completing the proof.

\begin{remark}\label{rem:swapxy1}
It might be more efficient to incorporate this step into the procedure described in
Section~\ref{sec:split}. In particular, we could show that $x,y$ can be interchanged
as soon as $v$ and $x$ have a common neighbour $w$. However, as we only need to
show that $\Mnd$ is connected, and not that we can find shortest paths in
$\Mnd$, we prefer to separate these two steps, for clarity.
\end{remark}

\begin{remark}\label{rem:swapxy2}
We require only one \ds\ to interchange $x,y$, but we may have to repeat this $d$
times.  Since $O(d^3)$ steps are needed to create a clique component
(see Remark~\ref{rem:bound}), this gives $O(d^4)$ steps in total for each
inductive step. This must be repeated in graphs of order $n-i(d+1)$ for
$0\leq i< \rdown{n/(d+1)}\rfloor$, that is, $O(n/d)$ times.
Thus in total we may need $O(nd^3)$ \ds es to connect $X$ with $Y$.

\end{remark}

\section{Regular graphs of degree at most two}\label{sec:d=2}

Theorem~\ref{thm:irreducible} excludes the cases $d=0,1,2$. The switch chain is irreducible in all these cases. We will briefly examine the question of
connectedness of $\Mnd$ (equivalently,
irreducibility of \ds\ chains on $\cG_{n,d}$) when $d=0,1,2$.

If $d=0$ then the unique graph in $\Gnd[n,0]$ is a labelled independent set of
order $n$. Hence $\Mnd[n,0]$ is trivially connected and any \ds\ chain 
is trivially irreducible. If $d=1$ then $G\in\cG_{n,1}$ is a matching and $n$ must be even. 
Now $|\cG_{n,1}| > 1$ when $n\geq 4$ is even,
but clearly no
\ds\ is possible as no element of $\cG_{n,1}$ contains a triangle or a path of
four edges. Thus $\Mnd[n,1]$ is not connected when $n\geq 4$ is even (indeed,
$\Mnd[n,1]$ has no edges in this case).

For $d=2$, it is not so obvious whether 
$\Mnd[n,2]$  is connected.
We will now deal with this, but first we will prove a relevant property of \ds es.

\begin{lemma}\label{permcycle}
Suppose that an induced cycle in a graph $G\in\Gnd$ has at least six vertices.
Then \ds es can be used to permute its vertices arbitrarily.
\end{lemma}
\begin{proof}
Any adjacent transposition can be simulated by a \dsp\ followed by a \dsm, as illustrated in Figure~\ref{fig:cycle}, where $v_1,v_2$ are transposed.
\begin{figure}[H]
\begin{center}
\begin{tikzpicture}[xscale=0.5,yscale=0.7,inner sep=0pt,font=\small]
\node (1) [b,label=left:$v_1\,$] at (-1.4,1) {};
\node (2) [b,label=above:${}_{}v_2$] at (0,2) {};
\node (3) [b,label=above:${}_{}v_3$] at (2,2) {};
\node (4) [b,label=above:${}_{}v_4$] at (4,2) {} ;
\node (n) [b,label=below: $\strut v_n$] at (0,0)  {};
\node (x) at (1.5,0) {}; \node (y) at (5,2)  {};
\draw  (x) -- (n) -- (1) -- (2)-- (3) -- (4) -- (y);
\draw [dashed] (1) -- (3) (4) -- (n) ;
\node at (6.5,1) {\Dp} ; 
\draw  (5.5,0.5)edge[line width=1.2pt, ->](7.25,0.5);
\begin{scope}[shift={(11,0)}]
\node (1) [b,label=left:$v_1\,$] at (-1.4,1) {};
\node (2) [b,label=above:${}_{}v_2$] at (0,2) {};
\node (3) [b,label=above:${}_{}v_3$] at (2,2) {};
\node (4) [b,label=above:${}_{}v_4$] at (4,2) {} ;
\node (n) [b,label=below: $\strut v_n$] at (0,0)  {};
\node (x) at (1.5,0) {}; \node (y) at (5,2)  {};
\draw  (x) -- (n) -- (4) -- (y) (1) -- (2)-- (3) -- (1) ;
\draw [dashed] (3) -- (4) (2) -- (n) ;
\node at (6.5,1) {\Dm} ; 
\draw  (5.5,0.5)edge[line width=1.2pt, ->](7.25,0.5);
\end{scope}
\begin{scope}[shift={(22,0)}]
\node (1) [b,label=left:$v_1\,$] at (-1.4,1) {};
\node (2) [b,label=above:${}_{}v_2$] at (0,2) {};
\node (3) [b,label=above:${}_{}v_3$] at (2,2) {};
\node (4) [b,label=above:${}_{}v_4$] at (4,2) {} ;
\node (n) [b,label=below: $\strut v_n$] at (0,0)  {};
\node (x) at (1.5,0) {}; \node (y) at (5,2)  {};
\draw  (3) -- (4) (x) -- (n) -- (2) -- (1) (4)-- (y)  (3) -- (1) ;
\end{scope}
\end{tikzpicture}
\end{center}\vspace{1ex}
  \caption{: Transposing two vertices on an induced cycle of length at least six}\label{fig:cycle}
\end{figure}
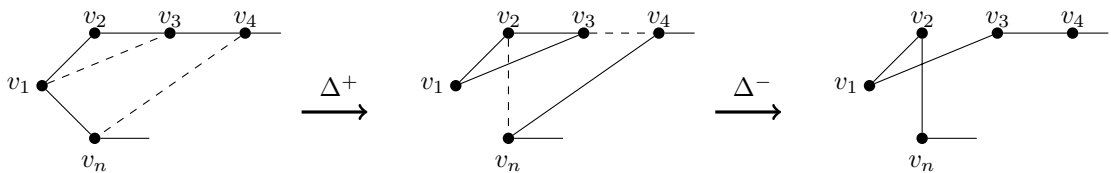
Then all vertices can be permuted as required using bubble-sort.
\end{proof}

If $n<3$ then $\Gnd[n,2]$ is empty. For $n\geq 3$, let $c_i$ ($i=1,2$)
be the number of cycles in $G\in\Gnd[n,2]$ such that their length modulo 3 is $i$.
Then $n\equiv c_1+2c_2 \pmod 3$. We will say $G$ has \emph{class} $(c_1,c_2)$.

If $n\geq 3$ then there is at least one class in $\Gnd[n,2]$.
If $n\equiv i \pmod 3$ ($i\in\{0,1,2\}$), then the class $(i,0)$ exists. Any class is preserved under \ds es, since a cycle can either be increased by length 3 by a \dsp\ or decreased by length 3 by a \dsm. Nothing else is possible. Thus two different classes cannot be connected by \ds es. But any graph in the class can be transformed by \dsm es to a ``canonical'' graph with $c_1$ 4-cycles, $c_2$ 5-cycles and $(n-4c_1-5c_2)/3$ triangles.
Note that there are $(k-1)!/2$ distinct labellings of the vertices of a $k$-cycle.

\begin{lemma}\label{d=2}
The graph $\Mnd[n,2]$ is connected if and only if $n\in\{3,6,7\}$.
Hence any given \ds\ chain is irreducible on $\cG_{n,2}$ if and only if $n\in \{3,6,7\}$.
\end{lemma}

\begin{proof}
Since \ds es do not change the class of a graph, it follows that
$\Mnd[n,2]$ is disconnected if $\cG_{n,2}$ contains
at least two distinct classes. Next observe that if there are two classes in $\cG_{n,2}$, then there are at least two classes in $\cG_{n+3k,2}$ for any $k\geq 0$.
Therefore, if there exists $\nu$ such
that two classes exist in $\cG_{n,2}$ for $n\in \{\nu, \nu+1,\nu+2\}$,
then there are at least two classes in $\cG_{n,2}$ for all $n\geq \nu$.
Taking $\nu=8$, we observe that
\begin{itemize}[itemsep=0pt, topsep=0pt]
\item $\cG_{8,2}$ has two classes: $(0,1)$ and $(2,0)$,
\item $\cG_{9,2}$ has two classes: $(0,0)$ and $(1,1)$,
\item $\cG_{10,2}$ has two classes: $(0,2)$ and $(1,0)$.
\end{itemize}
Thus there are at least two classes in $\cG_{d,2}$ for every $n\geq 8$, and hence
$\Mnd[n,2]$ is not connected for any $n\geq 8$. So no \ds\ chain on
$\Gnd[n,2]$ can be irreducible.

For $3\leq n\leq 7$ there is only one class: when $n=3$ the unique class is $(0,0)$,
for $n= 4$ it is $(1,0)$, for $n= 5$ it is $(0,1)$, for $n= 6$ it is $(0,0)$ and for $n= 7$ it is $(1,0)$. However, this is not enough to ensure that $\Mnd[n,2]$ is connected.
	
If $n=3$ then there is only one labelled cycle, so $\Mnd[3,2]$ is trivially connected.
There are three distinct labellings of a 4-cycle, and no way of
connecting them by \ds es. Thus $\Mnd[4,2]$ is not connected. Similarly,
there are 12 distinct labellings of a 5-cycle, and no possible \ds\
between them, so $\Mnd[5,2]$ is not connected.

However, for $n=6$ and $n=7$, any $G\in \Gnd[n,2]$ is either connected or contains
a triangle. 
So we can use a \dsm\ switch, if necessary, to ensure that the current
graph is an $n$-cycle.
Lemma~\ref{permcycle} then allows us to permute labels arbitrarily on the cycle,
which implies that $\Mnd[6,2]$ and $\Mnd[7,2]$ are both connected. 
That is, any \ds\ chain is irreducible on $\cG_{n,2}$ when $n=6$ or $n=7$.
\end{proof}

We remark that there can be more than two classes in $\cG_{n,2}$ when for $n>10$. For example, $\Gnd[20,2]$ contains the classes $(5,0)$ and $(0,4)$, as well as the two classes $(0,1)$ and $(2,0)$ inherited from $\Gnd[8,2]$.

\end{document}